\documentclass[USletter,     
    11pt,                   
    headings=normal,        
    dvipsnames,
]{article}%
\usepackage[margin=1in]{geometry}
\usepackage{cancel}
\usepackage[normalem]{ulem}

\PassOptionsToPackage{utf8}{inputenc}
\usepackage{inputenc}

\usepackage{amsmath}
\usepackage[USenglish]{babel}
\usepackage{amsthm, amsfonts, amssymb}
\usepackage{mathtools}
\usepackage{aligned-overset}
\usepackage{doi}
\usepackage{wrapfig}
\usepackage[noend, algosection]{algorithm2e}
\usepackage{paralist}
\usepackage{soul}

\usepackage[shortlabels]{enumitem}
\usepackage{thmtools}
\usepackage{comment}
\usepackage[dvipsnames]{xcolor}

\usepackage{tikz}

\hypersetup{
    pdftitle={The Strong Nine Dragon Tree Conjecture},    
    pdfsubject={Preprint},
    pdfauthor={Sebastian Mies},    
    plainpages=false,           
    colorlinks=false,           
    linkcolor=ctcolormain,      
    citecolor=ctcolormain,      
    pdfborder={0 0 0},          
    breaklinks=true,            
    bookmarksnumbered=true,     %
    bookmarksopen=true          %
}

\graphicspath{{images/}}

\definecolor{lightBlue}{RGB}{136, 247, 244}

\def\scale{0.5}
\def\lengthP{12}

\def\tikbegin{\begin{tikzpicture}[scale=\scale, baseline=-1mm, -, 
]}

\tikzset{
string/.style={},
	redDot/.style={circle, draw, red, fill=red, inner sep=1.5pt},
	blackDot/.style={circle, draw, fill, inner sep=1.5pt},
	transp/.style={inner sep=1.5pt},
	blueEdge/.style={line width = 0.5mm, cyan, ->, -stealth},
	redEdge/.style={red, line width=0.5mm},
	redDotted/.style={red, densely dotted, line width=0.5mm},
	blueDotted/.style={cyan, loosely dotted, line width=0.5mm, ->, -stealth},
}

\def\caseOneNodes{
	\node[redDot, label={[label distance=\offsetNames]90:$x$}] (x) at (0, 0) {};
	\node[redDot, label={[label distance=\offsetNames]90:$x_i$}] (xi) at (\lengthKx, 0) {};
	\node[redDot, label={[label distance=\offsetNames]90:$x_{i+1}$}] (xi1) at (\lengthKx + \lengthEdge, 0) {};
	
	\node[redDot, label={[label distance=\offsetNames]90:$x_j$}] (xj) at (\lengthKx + \lengthEdge + \lengthKThree, 0) {};
	\node[redDot, label={[label distance=\offsetNames]90:$x_{j+1}$}] (xj1) at (\lengthKx + 2*\lengthEdge + \lengthKThree, 0) {};
	\node[redDot, label={[label distance=\offsetNames]90:$y$}] (y) at (2*\lengthKx + 2*\lengthEdge + \lengthKThree, 0) {};
	
	\node[redDot, label={[label distance=\offsetNames]0:$x'$}] (x') at (0, \lengthEdge) {};
	\node[redDot, label={[label distance=\offsetNames]0:$y'$}] (y') at (2*\lengthKx + 2*\lengthEdge + \lengthKThree, \lengthEdge) {};
}

\def\caseOneNodesUV{
	\node[redDot, label={[label distance=\offsetNames]90:$x$}] (x) at (0, 0) {};
	\node[redDot, label={[label distance=\offsetNames]90:$u_x$}] (xi) at (\lengthKx, 0) {};
	\node[redDot, label={[label distance=\offsetNames]90:$u_y$}] (xi1) at (\lengthKx + \lengthEdge, 0) {};
	
	\node[redDot, label={[label distance=\offsetNames]90:$v_x$}] (xj) at (\lengthKx + \lengthEdge + \lengthKThree, 0) {};
	\node[redDot, label={[label distance=\offsetNames]90:$v_y$}] (xj1) at (\lengthKx + 2*\lengthEdge + \lengthKThree, 0) {};
	\node[redDot, label={[label distance=\offsetNames]90:$y$}] (y) at (2*\lengthKx + 2*\lengthEdge + \lengthKThree, 0) {};
	
	\node[redDot, label={[label distance=\offsetNames]0:$x'$}] (x') at (0, \lengthEdge) {};
	\node[redDot, label={[label distance=\offsetNames]0:$y'$}] (y') at (2*\lengthKx + 2*\lengthEdge + \lengthKThree, \lengthEdge) {};
}

\def\caseOnePic{

\def\lengthKx{2.5}
\def\lengthKThree{\lengthP - 2*\lengthKx - 2*\lengthEdge}
\def\offsetNames{-2em}
\def\lengthEdge{1.5}

\tikbegin
	\caseOneNodes

	\draw[redDotted]
		(x) edge (xi)
		(xi1) edge (xj)
		(xj1) edge (y)
		;
	\draw[redEdge]
		(xi) edge (xi1)
		(xj) edge (xj1)
		;
	\draw[blueEdge]
		(x) edge (x')
		(y) edge (y')
		;
	\draw[blueDotted]
		(xi) edge[bend right=50] (x)
		(xj1) edge[bend left=50] (y)
		;
\end{tikzpicture}
}

\def\caseOnePicAugmFalse{

\def\lengthKx{2.5}
\def\lengthKThree{\lengthP - 2*\lengthKx - 2*\lengthEdge}
\def\offsetNames{-2em}
\def\lengthEdge{1.5}

\tikbegin
	\caseOneNodesUV
	
	\node[transp, label={[text=red, label distance=1.5*\offsetNames]270:$K'$}] (K') at (\lengthKx+\lengthEdge, 0) {};
	\node[transp, label={[text=red, label distance=\offsetNames]90:$K''$}] (K'') at (1.5*\lengthKx + 2*\lengthEdge + \lengthKThree, 0) {};

	\draw[redDotted]
		(x) edge (xi)
		(xi1) edge (xj)
		(xj1) edge (y)
		;
	\draw[redEdge]
		(xi) edge (xi1)
		(y) edge (y')
		;
	\draw[blueEdge]
		(x) edge (x')
		(xj1) edge (xj)
		;
	\draw[blueDotted]
		(xi) edge[bend right=50] (x)
		(y) edge[bend right=50] (xj1)
		;
\end{tikzpicture}
}

\def\caseOnePicAugmTrue{

\def\lengthKx{2.5}
\def\lengthKThree{\lengthP - 2*\lengthKx - 2*\lengthEdge}
\def\offsetNames{-2em}
\def\lengthEdge{1.5}

\tikbegin
	\caseOneNodesUV
	
	\node[transp, label={[text=red, label distance=\offsetNames]90:$K'$}] (K') at (0.5*\lengthKx, 0) {};
	\node[transp, label={[text=red, label distance=1.5*\offsetNames]270:$K''$}] (K'') at (\lengthKx + \lengthEdge + \lengthKThree, 0) {};

	\draw[redDotted]
		(x) edge (xi)
		(xi1) edge (xj)
		(xj1) edge (y)
		;
	\draw[redEdge]
		(x) edge (x')
		(xj1) edge (xj)
		;
	\draw[blueEdge]
		(y) edge (y')
		(xi) edge (xi1)
		;
	\draw[blueDotted]
		(x) edge[bend left=50] (xi)
		(xj1) edge[bend left=50] (y)
		;
\end{tikzpicture}
}

\def\caseTwoNodes{
	\node[redDot, label={[label distance=\offsetNames]90:$x$}] (x) at (0, 0) {};
	\node[redDot, label={[label distance=\offsetNames]90:$x_i$}] (xi) at (\lengthKx, 0) {};
	\node[redDot, label={[label distance=\offsetNames]90:$x_{i+1}$}] (xi1) at (\lengthKx + \lengthEdge, 0) {};
	\node[redDot, label={[label distance=\offsetNames]90:$y$}] (y) at ({2*(\lengthKx) + \lengthEdge}, 0) {};
	
	\node[redDot, label={[label distance=\offsetNames]0:$x'$}] (x') at (0, \lengthEdge) {};
	\node[redDot, label={[label distance=\offsetNames]0:$y'$}] (y') at ({2*(\lengthKx) + \lengthEdge}, \lengthEdge) {};

}

\def\caseTwoPic{

\def\lengthEdge{1.5}
\def\lengthKx{\lengthP/2 - \lengthEdge / 2}
\def\offsetNames{-2em}

\tikbegin
	\caseTwoNodes
	
	\draw[redDotted]
		(x) edge (xi)
		(xi1) edge (y)
		;
	\draw[redEdge]
		(xi) edge (xi1)
		;
	\draw[blueEdge]
		(x) edge (x')
		(y) edge (y')
		;
	\draw[blueDotted]
		(xi) edge[bend right=20] (x)
		(xi1) edge[bend left=20] (y)
		(x') edge[bend left=10, in=150, out=30] (y)
		;
\end{tikzpicture}
}

\def\caseThreePic{

\def\lengthEdge{1.5}
\def\lengthKx{\lengthP/2 - \lengthEdge / 2}
\def\offsetNames{-2em}

\tikbegin
	\caseTwoNodes
	
	\draw[redDotted]
		(x) edge (xi)
		(xi1) edge (y)
		;
	\draw[redEdge]
		(xi) edge (xi1)
		;
	\draw[blueEdge]
		(x) edge (x')
		(y) edge (y')
		;
	\draw[blueDotted]
		(xi) edge[bend right=20] (x)
		(xi1) edge[bend left=20] (y)
		;
\end{tikzpicture}
}

\def\caseTwoNodesUV{
	\node[redDot, label={[label distance=\offsetNames]90:$x$}] (x) at (0, 0) {};
	\node[redDot, label={[label distance=\offsetNames]90:$u$}] (xi) at (\lengthKx, 0) {};
	\node[redDot, label={[label distance=\offsetNames]90:$v$}] (xi1) at (\lengthKx + \lengthEdge, 0) {};
	\node[redDot, label={[label distance=\offsetNames]90:$y$}] (y) at ({2*(\lengthKx) + \lengthEdge}, 0) {};
	
	\node[redDot, label={[label distance=\offsetNames]0:$x'$}] (x') at (0, \lengthEdge) {};
	\node[redDot, label={[label distance=\offsetNames]0:$y'$}] (y') at ({2*(\lengthKx) + \lengthEdge}, \lengthEdge) {};

}

\def\caseThreePicAugm{

\def\lengthEdge{1.5}
\def\lengthKx{\lengthP/2 - \lengthEdge / 2}
\def\offsetNames{-2em}

\tikbegin
	\caseTwoNodesUV
	
	\node[transp, label={[text=red, label distance=\offsetNames]90:$K'$}] (K') at (0.5*\lengthKx, 0) {};
	\node[transp, label={[text=red, label distance=\offsetNames]90:$K_2$}] (K2) at (1.5*\lengthKx + \lengthEdge, 0) {};
	
	\draw[redDotted]
		(x) edge (xi)
		(xi1) edge (y)
		;
	\draw[redEdge]
		(x) edge (x')
		;
	\draw[blueEdge]
		(y) edge (y')
		(xi) edge (xi1)
		;
	\draw[blueDotted]
		(x) edge[bend left=20] (xi)
		(xi1) edge[bend left=20] (y)
		;
\end{tikzpicture}
}

\def\specialPaths{
\def\scale{0.7}
\def\lengthEdge{1.5}
\def\offsetNames{-2em}
\def\dotOffset{0.35}

\tikbegin
\node[blackDot, label={[label distance=\offsetNames]90:$v_{-1}$}] (v-1) at 
	(0, 0) {};
\node[blackDot, label={[label distance=\offsetNames]90:$v_0$}] (v0) at 
	(\lengthEdge, 0) {};
\node[transp, label={[label distance=1.3*\offsetNames]270:\textcolor{lightBlue}{$T_{b_1}$} }] (threeDots1) at 
	(2*\lengthEdge + \dotOffset, 0) {...};
\node[blackDot, label={[label distance=\offsetNames]90:$v_{a_1 - 1}$}] (a1') at 
	(3*\lengthEdge + 2*\dotOffset, 0) {};

\node[blackDot, label={[label distance=\offsetNames]90:$v_{a_1}$}] (a1) at 
	(4*\lengthEdge + 2*\dotOffset, 0) {};
\node[transp, label={[label distance=1.3*\offsetNames]270:\textcolor{cyan}{$T_{b_2}$} }] (threeDots2) at 
	(5*\lengthEdge + 3*\dotOffset, 0) {...};
\node[blackDot, label={[label distance=\offsetNames]90:$v_{a_2 - 1}$}] (a2') at 
	(6*\lengthEdge + 4*\dotOffset, 0) {};

\node[blackDot, label={[label distance=\offsetNames]90:$v_{a_2}$}] (a2) at 
	(7*\lengthEdge + 4*\dotOffset, 0) {};
\node[transp, label={[label distance=1.3*\offsetNames]270:\textcolor{blue}{$T_{b_3}$} }] (threeDots3) at 
	(8*\lengthEdge + 5*\dotOffset, 0) {...};
\node[blackDot, label={[label distance=\offsetNames]90:$v_{a_3 - 1}$}] (a3') at 
	(9*\lengthEdge + 6*\dotOffset, 0) {};

\node[blackDot, label={[label distance=\offsetNames]90:$v_{a_3}$}] (a3) at 
	(10*\lengthEdge + 6*\dotOffset, 0) {};

\draw[redEdge]
	(v-1) edge (v0)
	;
\draw[line width = 0.5mm, lightBlue, ->, -stealth]
	(v0) edge (threeDots1)
	(threeDots1) edge (a1')
	(a1') edge (a1)
	;
\draw[blueEdge]
	(a1) edge (threeDots2)
	(threeDots2) edge (a2')
	(a2') edge (a2)
	;
\draw[line width = 0.5mm, blue, ->, -stealth]
	(a2) edge (threeDots3)
	(threeDots3) edge (a3')
	(a3') edge (a3)
	;
\draw[line width = 0.5mm, lightBlue, dotted]
	(a1) edge[bend right=50] (v-1);
\draw[line width = 0.5mm, cyan, dotted]
	(a2) edge[bend right=50] (a1');
\draw[line width = 0.5mm, blue, dotted]
	(a3) edge[bend right=50] (a2');
\end{tikzpicture}
}

\def\specialPathsAugm{
\def\scale{0.7}
\def\lengthEdge{1.5}
\def\offsetNames{-2em}
\def\dotOffset{0.35}

\tikbegin
\node[blackDot, label={[label distance=\offsetNames]90:$v_{-1}$}] (v-1) at 
	(0, 0) {};
\node[blackDot, label={[label distance=\offsetNames]90:$v_0$}] (v0) at 
	(\lengthEdge, 0) {};
\node[transp, label={[label distance=1.3*\offsetNames]270:\textcolor{lightBlue}{$T_{b_1}$} }] (threeDots1) at 
	(2*\lengthEdge + \dotOffset, 0) {...};
\node[blackDot, label={[label distance=\offsetNames]90:$v_{a_1 - 1}$}] (a1') at 
	(3*\lengthEdge + 2*\dotOffset, 0) {};

\node[blackDot, label={[label distance=\offsetNames]90:$v_{a_1}$}] (a1) at 
	(4*\lengthEdge + 2*\dotOffset, 0) {};
\node[transp, label={[label distance=1.3*\offsetNames]270:\textcolor{cyan}{$T_{b_2}$} }] (threeDots2) at 
	(5*\lengthEdge + 3*\dotOffset, 0) {...};
\node[blackDot, label={[label distance=\offsetNames]90:$v_{a_2 - 1}$}] (a2') at 
	(6*\lengthEdge + 4*\dotOffset, 0) {};

\node[blackDot, label={[label distance=\offsetNames]90:$v_{a_2}$}] (a2) at 
	(7*\lengthEdge + 4*\dotOffset, 0) {};
\node[transp, label={[label distance=1.3*\offsetNames]270:\textcolor{blue}{$T_{b_3}$} }] (threeDots3) at 
	(8*\lengthEdge + 5*\dotOffset, 0) {...};
\node[blackDot, label={[label distance=\offsetNames]90:$v_{a_3 - 1}$}] (a3') at 
	(9*\lengthEdge + 6*\dotOffset, 0) {};

\node[blackDot, label={[label distance=\offsetNames]90:$v_{a_3}$}] (a3) at 
	(10*\lengthEdge + 6*\dotOffset, 0) {};

\draw[redEdge]
	(a3') edge (a3)
	;
\draw[line width = 0.5mm, lightBlue, ->, -stealth]
	(v0) edge (v-1)
	(threeDots1) edge (v0) 
	(a1') edge(threeDots1)
	;
\draw[blueEdge]
	(a1) edge (a1')
	(threeDots2) edge(a1)
	 (a2') edge (threeDots2)
	;
\draw[line width = 0.5mm, blue, ->, -stealth]
	(a2) edge (a2')
	(threeDots3) edge (a2)
	(a3') edge (threeDots3)
	;
\draw[line width = 0.5mm, lightBlue, dotted]
	(a1) edge[bend right=50] (v-1);
\draw[line width = 0.5mm, cyan, dotted]
	(a2) edge[bend right=50] (a1');
\draw[line width = 0.5mm, blue, dotted]
	(a3) edge[bend right=50] (a2');
\end{tikzpicture}
}

\def\xyzNodes{
	\node[redDot, label={[label distance=\offsetNames]90:$x$}] (x) at (0, 0) {};
	\node[redDot, label={[label distance=\offsetNames]90:$u$}] (u) at (\lengthKx + \lengthEdge, 0) {};
	\node[redDot, label={[label distance=\offsetNames]90:$y$}] (y) at (\lengthKx + 2*\lengthEdge, 0) {};
	
	\node[redDot, label={[label distance=\offsetNames]0:$x'$}] (x') at (0, \lengthEdge) {};
	\node[redDot, label={[label distance=\offsetNames]180:$y'$}] (y') at (\lengthKx + 3*\lengthEdge, 0) {};
	
	\node[redDot, label={[label distance=\offsetNames]0:$v$}] (v) at 
	({0.3*(\lengthKx)}, {-\tv}) {};
	\node[redDot, label={[label distance=\offsetNames]0:$z$}] (z) at 
	({0.3*(\lengthKx)}, {-\tv - \lengthEdge}) {};
	\node[redDot, label={[label distance=\offsetNames]180:$z'$}] (z') at 
	({0.3*(\lengthKx)}, {-\tv - 2*\lengthEdge}) {};
	
	\node[redDot] (x'') at (0, 2*\lengthEdge) {};
	\node[redDot] (y'') at (\lengthKx + 3*\lengthEdge, \lengthEdge) {};
	\node[redDot] (z'') at ({0.3*(\lengthKx) - \lengthEdge}, {-\tv - 2*\lengthEdge}) {};

	\node[transp] (t) at 
	({0.3*(\lengthKx)}, 0) {};
	\node[transp] (s) at 
	({0.5*(\lengthKx)}, 0.5*\lengthEdge) {};
}

\def\xyzPicOne{
	\def\lengthP{6}
	\def\lengthEdge{1.5}
	\def\lengthKx{\lengthP - \lengthEdge}
	\def\offsetNames{-2em}
	\def\tv{0.5*\lengthP}
	
	\tikbegin
	\xyzNodes
	\draw[redDotted]
	(x) edge (u)
	(u) edge (y)
	(t) edge (v)
	;
	\draw[redEdge]
	(u) edge (y)
	(v) edge (z)
	(x') edge (x'')
	(y') edge (y'')
	(z') edge (z'')
	;
	\draw[blueEdge]
	(x) edge (x')
	(y) edge (y')
	(z) edge (z')
	;
	\draw[blueDotted]
	(u) edge[bend right=10] (s)
	(s) edge[bend right=10] (x)
	(v) edge[bend right=20] (s)
	(x') edge[bend left=20] (y)
	(y') edge[bend left=20] (z)
	;
	
	\draw[blueDotted]
	;

\end{tikzpicture}
}

\def\xyzPicTwo{
	\def\lengthP{6}
	\def\lengthEdge{1.5}
	\def\lengthKx{\lengthP - \lengthEdge}
	\def\offsetNames{-2em}
	\def\tv{0.5*\lengthP}
	
	\tikbegin
	\xyzNodes
	\draw[redDotted]
	(x) edge (u)
	(u) edge (y)
	(t) edge (v)
	;
	\draw[redEdge]
	(v) edge (z)
	(x') edge (x'')
	(y') edge (y'')
	(z') edge (z'')
	(x) edge (x')
	;
	\draw[blueEdge]
	(y) edge (y')
	(z) edge (z')
	(u) edge (y)
	;
	\draw[blueDotted]
	(s) edge[bend left=10] (u)
	(x) edge[bend left=10] (s)
	(v) edge[bend right=20] (s)
	(x') edge[bend left=20] (y)
	(y') edge[bend left=20] (z)
	;
	
	\draw[blueDotted]
	;
	
\end{tikzpicture}
}

\def\xyzPicThree{
	\def\lengthP{6}
	\def\lengthEdge{1.5}
	\def\lengthKx{\lengthP - \lengthEdge}
	\def\offsetNames{-2em}
	\def\tv{0.5*\lengthP}
	
	\tikbegin
	\xyzNodes
	\draw[redDotted]
	(x) edge (u)
	(u) edge (y)
	(t) edge (v)
	;
	\draw[redEdge]
	(x') edge (x'')
	(y') edge (y'')
	(z') edge (z'')
	(x) edge (x')
	(y) edge (y')
	;
	\draw[blueEdge]
	(z) edge (z')
	(y) edge (u)
	(v) edge (z)
	;
	\draw[blueDotted]
	(u) edge[bend right=10] (s)
	(x) edge[bend left=10] (s)
	(s) edge[bend left=20] (v)
	(x') edge[bend left=20] (y)
	(y') edge[bend left=20] (z)
	;
	
	\draw[blueDotted]
	;
	
\end{tikzpicture}
}

\usepackage{todonotes}
\usepackage[style=numeric-comp, maxbibnames=99]{biblatex}

\author{Sebastian Mies\thanks{Institute of Computer Science, Johannes Gutenberg University Mainz, email: smies@students.uni-mainz.de}     \, and Benjamin Moore\thanks{Charles University, Institute of Computer Science, Prague,  supported by project 22-17398S (Flows and cycles in graphs on surfaces) of Czech Science Foundation. Email: brmoore@iuuk.mff.cuni.cz.}}

\title{The Strong Nine Dragon Tree Conjecture is True for $d \leq k + 1$ }

\bibliography{bib-refs.bib}

\DeclarePairedDelimiter\ceil{\lceil}{\rceil}

\newcommand\rP{\prescript{r}{}\!P}

\newcommand\fracArb{\gamma}

\newcommand\explSG{H_{\mathcal T}}
\newcommand\density{\chi(d, k)}

\newcommand\sigGreater[1]{\sigma(>\!\!#1)}
\newcommand\sigStarGreater[1]{\sigma^*(>\!\!#1)}

\DeclareRobustCommand\caseOne{\overset{(1)}{\rightarrow}}
\DeclareRobustCommand\caseOneUndir{\overset{(1)}{\text{---}}}
\DeclareRobustCommand\caseTwo{\overset{(2)}{\rightarrow}}
\DeclareRobustCommand\caseThree{\overset{(3)}{\text{---}}}

\newcommand\iSig{i_{\sigma^*}}

\newcommand\KC{K_{\mathcal C}}

\declaretheoremstyle[%
  spaceabove=-2em,%
  spacebelow=6pt,%
  headfont=\normalfont\itshape,%
  postheadspace=1em,%
  qed=\qedsymbol%
]{mystyle}
\declaretheoremstyle[%
  spaceabove=-2em,%
  spacebelow=6pt,%
  headfont=\normalfont\itshape,%
  postheadspace=1em,%
  qed=\hfill \textit{\color{gray}(End of proof of the claim) }$\;\blacksquare$%
]{proofInProofStyle}

\declaretheorem[name={Proof},style=mystyle,unnumbered,
]{beweis}

\newtheorem{thm}{Theorem}[section] 

\newtheorem{lemma}[thm]{Lemma}
\newtheorem{satz}[thm]{Theorem}

\newtheorem{conj}[thm]{Conjecture}
\newtheorem{definition}[thm]{Definition}
\newtheorem{beob}[thm]{Observation}
\newtheorem{kor}[thm]{Corollary}

\newtheorem{notation}[thm]{Notation}

\newtheorem*{ack}{Acknowledgements}
\date{}
\begin{document}

\maketitle

\begin{abstract} 
     The arboricity $\Gamma(G)$ of an undirected graph $G = (V,E)$ is the minimal number $k$ such that $E$ can be partitioned into $k$ forests. Nash-Williams' formula states that $k = \ceil{ \fracArb(G) }$, where $\fracArb(G)$ is the maximum of $|E_H|/(|V_H| -1)$ over all subgraphs $(V_H, E_H)$ of $G$ with $|V_H| \geq 2$.
    
    The Strong Nine Dragon Tree Conjecture states that if $\fracArb(G) \leq k + \frac{d}{d+k+1}$ for $k, d \in \mathbb N_0$, then there is a partition of the edge set of $G$ into $k+1$ forests such that one forest has at most $d$ edges in each connected component.
    
    We settle the conjecture for $d \leq k + 1$. For $d \leq 2(k+1)$, we cannot prove the conjecture, however we show that there exists a partition in which the connected components in one forest have at most $d + \ceil{k \cdot \frac{d}{k+1}} - k$ edges.
    
    As an application of this theorem, we show that every $5$-edge-connected planar graph $G$ has a $\frac{5}{6}$-thin spanning tree. This theorem is best possible, in the sense that we cannot replace $5$-edge-connected with $4$-edge-connected, even if we replace $\frac{5}{6}$ with any positive real number less than $1$. This strengthens a result of Merker and Postle which showed $6$-edge-connected planar graphs have a $\frac{18}{19}$-thin spanning tree.
\end{abstract}
\SetAlgorithmName{Algorithm}{algorithm}{List of Heuristics}

\section{Introduction}

In this paper, graphs may have parallel edges, but no loops.
The celebrated Nash-Williams' Theorem states that a graph $G$ has a decomposition into $k$ forests if and only if $\fracArb(G) \leq k$ where 
$\fracArb(G) = \max_{H \subseteq G, v(H) \geq 2} \frac{e(H)}{v(H) - 1}$.
 Recall that a \textit{decomposition} is a partitioning of the edge set of a graph into subgraphs. Also, we use the notation that $e(H) = |E(H)|$ and $v(H) = |V(H)|$. The smallest $k$ for which such a decomposition exists is called the \textit{arboricity} of $G$, and hence naturally we call $\fracArb(G)$ the \textit{fractional arboricity} of $G$. Thus there is a connection between the edge density of subgraphs and the arboricity of $G$. Observe that if $\fracArb(G) = k + \varepsilon$ holds for $k \in \mathbb N$ and a small $\varepsilon > 0$ it suffices to remove only a few edges from the densest subgraphs of $G$ to make the resulting graph decompose into $k$ forests. Thus one could guess that we can find decompositions of $G$ into $k + 1$ forests where one has ``structure" depending on how small $\varepsilon$ is. As an example, Gonçalves showed in \cite{planar24} that planar graphs can be decomposed into three forests such that one forest has maximum degree at most 4. The Nine Dragon Tree Conjectures ask for similar structure as in Gonçalves' theorem, however in the more general situation where we only have information about the fractional arboricity:

\begin{satz}[Nine Dragon Tree Theorem \cite{ndtt}]	\label{satz:ndtt}
Let $G$ be a graph and $k$ and $d$ be positive integers. If $\fracArb(G) \leq k + \frac{d}{d + k + 1}$, then there is a decomposition into $k + 1$ forests, where one of the forests has maximum degree at most $d$.
\end{satz}

\begin{conj}[Strong Nine Dragon Tree Conjecture \cite{sndtck1d2}]
Let  $G$ be a graph and $d, k \in \mathbb N_0$. If $\fracArb(G) \leq k + \frac{d}{d + k + 1}$ then there is a partition into  $k + 1$ forests, where in one forest every connected component has at most $d$ edges.
\end{conj}

Both statements were proposed by Montassier et. al \cite{sndtck1d2}, who proved the $(k,d)=(1,1)$ and $(k,d) = (1,2)$ case of the Nine Dragon Tree Theorem. Prior to the full solution, various partial results were obtained towards the Nine Dragon Tree Theorem, for instance the $d=1$ case \cite{Yangmatching} and the case where $k \leq 2$ \cite{ndtk2,Kostochkaetal}, before the Nine Dragon Tree Theorem was proven by a beautiful argument of Jiang and Yang \cite{ndtt}. Note that when $d = 1$, the Nine Dragon Tree Theorem implies the Strong Nine Dragon Tree Conjecture, and hence the Strong Nine Dragon Tree Conjecture is known when $d=1$. Besides this case, prior to this paper, the Strong Nine Dragon Tree Conjecture has only been proven when  $(k, d) = (1, 2)$ \cite{Kostochkaetal}.
It is important to note that the Nine Dragon Tree Theorem is best possible in the following sense:

\begin{satz}[\cite{sndtck1d2}]
For any positive integers $k$ and $d$ there are arbitrarily large graphs $G$ and a set $S = 
\{e_{1},\ldots,e_{d+1}\} \subseteq E(G)$ of $d+1$ edges such that $\fracArb(G-S) = k + \frac{d}{k+d+1}$ and $G$ does not decompose into $k+1$ forests where one of the forests has maximum degree $d$. 
\end{satz}

Our contribution is a proof of the Strong Nine Dragon Tree Conjecture when $d \leq k+1$, and bounds when $k + 1 < d \leq 2(k+1)$:

\begin{satz}\label{satz:mySndtc}
For $d \leq k + 1$ the Strong Nine Dragon Tree Conjecture is true. If $k + 1 < d \leq 2(k + 1)$ then for any graph $G$ with fractional arboricity at most $k + \frac{d}{k+d+1}$, there is a decomposition into $k + 1$ forests where in one forest every connected component has at most $d + \ceil{k \cdot \frac{d}{k+1}} - k$ edges.
\end{satz}

We note as shown in \cite{sndtck1d2}, Theorem \ref{satz:mySndtc} implies the stronger statement where for any vertex $v \in V(G)$, there is a decomposition satisfying the outcome of Theorem \ref{satz:mySndtc}, and $v$ is an isolated vertex in the forest with bounded component size.

This is the first case of the Strong Nine Dragon Tree Conjecture known when $d \neq 1$ and $(k,d) \neq (1,2)$. We do note however there are some special cases of the Strong Nine Dragon Tree Conjecture that are known. For instance, an approximation scheme by Blumenstock and Fischer \cite{approxArb} transforms a decomposition of $k$ pseudoforests of a simple graph into $k + 1$ forests with one forest having at most $k$ edges in each component. Here, recall that a \textit{pseudoforest} is a graph where each connected component contains at most one cycle. This implies that the Strong Nine Dragon Tree Conjecture is true for $d = k$ for all graphs where $\Gamma(G) - 1$ is the minimum number of pseudoforests needed in a pseudoforest decomposition of $G$. The use of pseudoforests as an approach to the Nine Dragon Tree Theorem is well documented. In fact, in a pivotal paper, Fan et al. \cite{ndttPsfs} proved a pseudoforest analogue of the Nine Dragon Tree Theorem, and the proof technique used was essential to the eventual resolution of the Nine Dragon Tree Theorem. More recently, a pseudoforest analogue of the Strong Nine Dragon Tree Conjecture was proven \cite{sndtcPsfs}, and a digraph version of the Nine Dragon Tree Conjecture was also proposed \cite{digraphndt}.

As an application of Theorem \ref{satz:mySndtc} we partially resolve an (implicit) conjecture in \cite{boundeddiameter}. They asked if there exists a constant $d$ such that every planar graph of girth at least five decomposes into two forests $F_{1},F_{2}$ such that every component of $F_{i}$ has diameter at most $d$. This conjecture implies the following weaker conjecture about $\varepsilon$-thin trees.

\begin{definition}
Let $\varepsilon \in (0,1)$ be a real number. Let $G$ be a connected graph, and $T$ a spanning tree of $G$. We say that $T$ is an $\varepsilon$-thin tree if for every cut-set $S \subseteq E(G)$, we have 
\[ \frac{|E(T) \cap S|}{|S|} \leq \varepsilon\]
\end{definition}

\begin{conj}[\cite{boundeddiameter}]
\label{thintreeconjecture}
There exists an $\varepsilon \in (0,1)$ such that every $5$-edge-connected planar graph admits two edge-disjoint $\varepsilon$-thin trees. 
\end{conj}

Thomassen showed that one cannot hope to strengthen $5$-edge-connectivity to $4$-edge-connectivity, as there is no $\varepsilon \in (0,1)$ such that every planar $4$-edge-connected graph has even a single $\varepsilon$-thin tree (see \cite{boundeddiameter} for a discussion). Towards the conjecture, we prove:

\begin{satz}
\label{thintrees}
Every $5$-edge-connected planar graph admits a $\frac{5}{6}$-thin tree. 
\end{satz}

The proof of this follows the ideas in \cite{boundeddiameter}, and the needed new ingredient is that planar graphs of girth at least five decompose into two forests, one of which has each component containing at most five edges.

Our proof of Theorem \ref{satz:mySndtc} follows the general framework developed in \cite{ndttPsfs,sndtcPsfs,ndtt,Yangmatching}. First, we show that if Theorem \ref{satz:mySndtc} is not true, there is a vertex minimal counterexample that has a decomposition into $k$ spanning trees and another forest. Then, over all decompositions into $k$ spanning trees and a forest, we choose one decomposition that is as close as possible to satisfying Theorem \ref{satz:mySndtc} with respect to a certain function which we define later. After this, we follow the reconfiguration approach used in the proof of the Nine Dragon Tree Theorem, where we try and massage our decomposition into one which satisfies the theorem.
If a component of the special forest with too many edges is ``near" to a component with few edges, we argue that we can exchange some edges between the forests in certain cases to either decrease the size of the big component, or get ``closer" to being able to reduce the size of the big component.
If there are no small components next to a big component, we show how to gradually move edges away from neighboring components until we end up in a situation where we have a big component next to small components. Repeating this procedure will eventually result in the theorem. While this techinque has been used many times, the advancement in the present paper is to dig deeper into how the exchanges of tree edges can actually behave. This allows us to gain more control over the exchanges, which allows to reduce the number of ``small" components near big components over the previous papers.

The paper is organized as follows. In Section 2 we give all of the basic definitions, as well as say how we pick the minimal counterexample. In Section 3 we go over the special path techinque from the Nine Dragon Tree Theorem and note some critical corollaries of this techinque. In Section 4 we describe how we will exchange edges between trees. In Sections 5 and 6 we build structural lemmas to bound the number of small components near big components. In Section 7 we use these tools to prove the Strong Nine Dragon Tree Conjecture when $d \leq k+1$. In Section 8 we prove Theorem \ref{satz:mySndtc}. In Section 9 we prove Theorem \ref{thintrees}.

\section{Defining the counterexample}

The goal of this section is to set up everything we need to define our minimal counterexample. First we pin down some basic notation. For a path $P$ with $k$ vertices, we will write $P = [v_{1},\ldots,v_{k}]$ where $v_{i}v_{i+1}$ is an edge for all $i \in \{1,\ldots,k-1\}$. Given two paths $P_{1} = [v_{1},\ldots,v_{k}]$ and $P_{2} = [v_{k},\ldots,v_{t}]$, we let $P_{1} \oplus P_{2}$ be the concatenation of the two paths. As we will also consider digraphs, we will use the notation $(u,v)$ is a directed edge from $u$ to $v$, and we extend the above notation for paths to directed paths if directions are used. If $P = [v_1, ..., v_k]$ is a directed path we denote $\rP = [v_k, ..., v_1]$ to be the path where all edges of $P$ are reoriented. For undirected graphs, we use the notation $\{u,v\}$ to denote an edge $uv$.

Rather than prove Theorem \ref{satz:mySndtc}, we will show the following:
\begin{satz}	\label{satz:hilfsbeh}
Let $G$ be a graph, $d, k \in \mathbb N_0$ and suppose that $\fracArb(G) \leq k + \density$ where
\[\density = 
	\begin{cases}
		\frac{d}{d+k+1} \text{, if } d \leq k + 1, \\
		\frac{d + k}{d+3k+1} \text{, if } k + 1 < d < 3(k + 1),
	\end{cases}
\] \\
then there is a decomposition of $G$ into $k+1$ forests where in one forest each connected component has at most $d$ edges.
\end{satz}

Our first point of order is to show this theorem implies Theorem \ref{satz:mySndtc}.

\begin{lemma}
Theorem~\ref{satz:hilfsbeh} implies Theorem~\ref{satz:mySndtc}.
\end{lemma}
\begin{beweis}
Let $k+1 < d \leq 2(k+1)$. It is easy to compute that
$k + 1 < d + \ceil{k \cdot \frac{d}{k+1}} - k < 3(k+1)$. Theorem~\ref{satz:mySndtc} now follows directly, since
\[
    \frac{d + \ceil{k \cdot \frac{d}{k+1}} - k + k}{d + \ceil{k \cdot \frac{d}{k+1}} - k + 3k + 1}
    \geq \frac{\frac{2k+1}{k+1}d}{\frac{2k+1}{k+1}d + 2k + 1}
    = \frac{d}{d+k+1}.
\]

\end{beweis}

The next observation gives simple bounds on $\chi(d,k)$, we omit the proof as it is simply rearranging equations.
\begin{beob}	\label{beob:densityInterval}
Let $k,d \in \mathbb{N}$. Then
\begin{enumerate}[(a)]	
	\item $\density \leq \frac{d}{d+k+1}$.
	\item If $d \leq k + 1$, then $\density \in [0, \frac{1}{2}]$.
	\item If $k + 1 < d < 3(k + 1)$, then $\density \in (\frac{1}{2}, \frac{2}{3}]$.
\end{enumerate}
\end{beob}

For the rest of the paper we fix integers $k,d \in \mathbb{N}$, and always assume that we have a graph $G$ which is a counterexample with minimum number of vertices to Theorem \ref{satz:hilfsbeh}. The first observation we need is that $G$ decomposes into $k$ spanning trees and another forest. This fact follows from a minor tweak to the proof to Lemma~2.1 of \cite{ndtt} so we omit the proof:

\begin{lemma}[\cite{ndtt}] \label{lemma:minimalCounterExample}
Every graph $G$ that is a vertex minimal counterexample to Theorem \ref{satz:hilfsbeh} admits a decomposition into forests $T_1, \dots, T_k, F$ such that $T_1, \dots, T_k$ are spanning trees.
\end{lemma}

Note, if $G$ decomposes into $k$ spanning trees $T_{1},\ldots,T_{k}$, and a forest $F$, it follows that $F$ is disconnected. Otherwise, $\fracArb(G) = k+1$, but $\fracArb(G) < k+1$, a contradiction.

Given a decomposition of $G$, we will want to measure how close it is to satisfying Theorem \ref{satz:hilfsbeh}. This is captured in the next definition:

\begin{definition}
The \textit{residue function} $\rho(F)$ of a forest $F$ is defined as the tuple $(\rho_{V(G)-1}(F), \\  \rho_{V(G)-2}(F), \ldots, \rho_{d+1}(F))$, where $\rho_i(F)$ is the number of components of $F$ having $i$ edges.
\end{definition}

We will want to compare residue function values of different forests using lexicographic ordering and want to find the decomposition with one forest minimizing the residue function.

\begin{notation}
Over all decompositions into $k$ spanning trees, $T_{1},\ldots,T_{k}$, and a forest $F$ we choose one where $F$ minimizes $\rho$ with respect to lexicographic order. We call this minimum tuple $\rho^*$. This forest $F$ has a component $R^*$ containing more than $d$ edges.
We choose a vertex $r \in V(R^*)$ of degree at least $2$ in $R^*$. Further, if $d > k+1$ (and thus $e(R^*) \geq 4$), then we choose $r$ such that it has degree $3$ in $R^*$ or if this is not possible, such that there are two edge disjoint paths in $R^*$  of length at least $2$ starting at $r$. We fix $R^*$ and $r$ for the rest of the paper.
\end{notation}

\begin{definition}
We define $\mathcal F$ to be the set of decompositions into forests $(T_1, \dots, T_k, F)$ of $G$ such that $T_1, \dots, T_k$ are spanning trees of $G$; $R^*$ is a connected component of the undirected forest $F$ and the edges of $T_1, \ldots T_k$ are directed towards $r$. We let $\mathcal F^* \subseteq \mathcal F$ be the set of decompositions $(T_{1},\ldots,T_{k},F) \in \mathcal F$ such that $\rho(F) = \rho^*$.
\end{definition}

The next definition is simply to make it easier to talk about decompositions in $\mathcal F$.

\begin{definition}
Let $\mathcal{T} = (T_1, \dots, T_k, F) \in \mathcal F$. We say that the (directed) edges of $T_1, \dots, T_k$ are \textit{blue edges} and the (undirected) edges of $F$ are \textit{red edges}. We define $E(\mathcal{T}) := E(T_1) \cup \dots \cup E(T_k) \cup E(F)$. For a subgraph $U \subseteq (V, \; E(\mathcal{T}))$ we write $E_b(U)$ and $E_r(U)$ for the set of blue and red edges of $U$, respectively. Furthermore, we write $e_b(U) = |E_b(U)|$ and $e_r(U) = |E_r(U)|$.
\end{definition}

Finally, we can define the critical subgraph which we will focus on for the rest of the paper:

\begin{definition} \label{def:explSubgraph}
Let $\mathcal{T} \in \mathcal F$.
The \textit{exploration subgraph} $\explSG$ of $\mathcal{T}$ is the subgraph of $(V, \; E(\mathcal{T}))$, where the vertex set $V(\explSG)$ consists of all vertices $v$ for which there is a sequence of vertices $r=x_1, \dots, x_l = v$ such that for all $1 \leq i < l$ it holds: $(x_i, x_{i+1}) \in E_b(\mathcal{T})$ or 
$\{x_i, x_{i+1}\} \in E_r(\mathcal{T})$, and the set of edges of $\explSG$ is defined as
\[E(\explSG) = \big\{ \{x, y\} \in E_r(\mathcal T) \: | \: x, y \in V(\explSG) \big\} \cup \big\{ (x, y) \in E_b(\mathcal{T}) \: | \: x, y \in V(\explSG) \big\}.\]
We also call a connected component of $(V, E_r(\mathcal T))$ a red component.
\end{definition}

The next observation shows the importance of the exploration subgraph:
\begin{beob}	\label{beob:densityExplSG}
Let $\mathcal{T} \in \mathcal F$. Then
\[\frac{e_r(\explSG)}{v(\explSG) - 1} \leq \density.\]
\end{beob}
\begin{beweis}
We have $e_b(\explSG) = k (v(\explSG) - 1)$, since each vertex $v \in V(\explSG) - r$ has exactly $k$ blue outgoing edges and $r$ has no blue outgoing edge in $H_{\mathcal{T}}$. We conclude:
\begin{align*}
k + \density & \geq \fracArb(G) \\
    &\geq \frac{e(\explSG)}{v(\explSG) - 1} \\
	&= \frac{e_b(\explSG)}{v(\explSG) - 1} + \frac{e_r(\explSG)}{v(\explSG) - 1} \\
	&= k + \frac{e_r(\explSG)}{v(\explSG) - 1.}
\end{align*}
\end{beweis}

In light of Observation \ref{beob:densityExplSG}, we will want to focus on the subgraphs of $F$ with low edge density:

\begin{definition}
A red component $K$ is \textit{small} if $e(K) = 0$ and $d \leq k + 1$, or if $e(K) \leq 1$ and $k + 1 < d < 3(k+1)$.
\end{definition}

Note that $K$ is small, if and only if $\frac{e(K)}{v(K)} < \density$.

\begin{beob} \label{beob:sizeOfSmall}
Let $K_1, K_2$ be small red components. Then $e(K_1), e(K_2) < \frac{d}{k+1}$ and $e(K_1) + e(K_2) \leq d - 1.$
\end{beob}

Now we turn our focus to the notion of legal orders, which is an ordering of components of $F$ that loosely tells us in what order we should augment the decomposition.

\begin{definition} \label{def:legalOrder}
Let $(T_{1},\ldots T_{k},F) = \mathcal{T} \in \mathcal F$. Let $\sigma = (R_{1},\ldots,R_{t})$ be a sequence of all red components in $\explSG$. We say $\sigma$ is a \textit{legal order} for $\mathcal T$ if $R_1 = R^*$, and further for each $1 < j \leq t$, there is an $i_j < j$ such that there is a blue directed edge $(x_j, y_j)$ with $x_j \in V(R_{i_j})$ and $y_j \in V(R_j)$. 
\end{definition}

It will be useful to compare legal orders, and we will again do so using the lexicographic ordering.

\begin{definition} 
Let $\mathcal T, \mathcal T' \in \mathcal F$. Suppose $\sigma = (R_1, \dots, R_t)$ and $\sigma' = (R'_1, \dots, R'_{t'})$ are legal orders for $\mathcal T$ and $\mathcal T'$, respectively. We say $\sigma$ is \textit{smaller than} $\sigma'$, denoted $\sigma < \sigma'$, if $(e(R_1), \dots, e(R_t))$ is lexicographically smaller than $(e(R'_1), \dots, e(R'_{t'}))$. If $t \neq t'$ we extend the shorter sequence with zeros in order to make the orders comparable.
\end{definition}

To make it easier to discuss legal orders, we introduce some more vocabulary:

\begin{definition}
Suppose $\sigma = (R_1, \dots, R_t)$ is a legal order for $\mathcal{T} \in \mathcal F$. We write $i_\sigma(v) := j$ for $v \in V(\explSG)$, if $v \in V(R_j)$.
We define $i_\sigma(H) := \min \{i_\sigma(v) | v \in V(H)\}$ for each subgraph $H \subseteq \explSG$. We also define $\sigGreater{v} := \{w \in V(\explSG) | i_\sigma(w) > i_\sigma(v)\}$. We say that  $R_i$ is a \textit{parent} of $R_j$ with respect to $\sigma$, if $i < j$ holds and if there is a blue edge $(x, y)$ with $x \in V(R_i), y \in V(R_j)$. In this case we also call $R_j$ a child of $R_i$ with respect to $\sigma$.
\end{definition}

Note that in the above definition, a component may have many parents, and further aside from $R^*$, all red components have a parent. For the purposes of tiebreaking how we pick legal orders, we introduce the next graph:

\begin{definition}
Let $(T_{1},\ldots,T_{k},F) = \mathcal{T} \in \mathcal F$ and let $\sigma = (R_{1},\ldots,R_{t})$ be a legal order for $\mathcal T$. 
Compliant to Definition \ref{def:legalOrder} we choose a blue edge $(x_j, y_j)$ for all $1 < j \leq t$. There might be multiple possibilities for this, but we simply fix one choice for $\sigma$. We then denote $T_\sigma := (V(\explSG), \: E_r(\explSG) \cup \{(x_j, y_j) \: | \: 1 < j \leq t\})$ which defines a tree that we call the \textit{auxiliary tree} of $\sigma$. We always consider $T_\sigma$ to be rooted at $r$.
\end{definition}

With this, we are in position to define our counterexample. As already outlined, $G$ is a vertex-minimal counterexample to the theorem. Further, we pick a legal order $\sigma^*$ for a decomposition $\mathcal T^* = (T^*_1, ..., T^*_{k}, F^*) \in \mathcal F^*$ such that there is no legal order $\sigma$ with $\sigma < \sigma^*$ for any $\mathcal T' \in \mathcal F^*$. We will use these notations for the minimal legal order and decomposition throughout the rest of the paper.

\section{Augmenting Special Paths}
In this section we consider the first method to find a smaller legal order or shrink a component with more than $d$ edges without creating any new components with more than $d$ edges.

This method from \cite{ndtt} roughly works the following way:
if a blue edge $e$ connecting two red components can be colored red without increasing the residue function, then in certain cases we can find a red edge $e'$ that can be colored blue in exchange. In order to find this edge we need to look for a certain blue directed path that ends at $e$ and starts at $e'$ and $e'$ has to be closer to $R^*$ with respect to the legal order than $e$ (here we are viewing $e$ as a subgraph with two endpoints to make sense of the term close). First, we formalize the requirements for such a blue path:

\begin{definition} 
Let $\sigma = (R_1, \dots, R_t)$ be a legal order for $\mathcal{T} \in \mathcal F$. We call a blue directed path $P=[v_0, v_1, \dots, v_{l}] \subseteq (V(\explSG), E_b(\explSG))$ \textit{special with respect to $\sigma$ and $(v_{l-1}, v_l)$} if $v_l \in \sigGreater{v_0}$. \\
For two special paths $P=[v_0, v_1, \dots, v_{l}]$ and $P'=[v'_0, v'_1, \dots, v'_{l'}]$ with respect to $\sigma$ and $(v_{l-1}, v_l)$ we write $P \leq P'$ if $i_\sigma(v_0) < i_\sigma(v'_0)$ or if $i_\sigma(v_0) = i_\sigma(v'_0)$ and $v_0$ in $T_\sigma$ is an ancestor  (with respect to the root $r$) of $v'_0$. We call a special path $P$ with respect to $\sigma$ and $(x, y)$ \textit{minimal (with respect to $\sigma$ and $(x, y)$)} if there is no special path $P' \neq P$ with respect to $\sigma$ and $(x, y)$ with $P' \leq P$.
\end{definition}

Note that if we have a minimal special path $P=[v_0, v_1, \dots, v_l]$ with respect to $\sigma$ and $(v_{l-1}, v_l)$ we have $v_0 \neq r$ because $r$ has no outgoing blue edge by construction. Therefore, $v_0$ has a parent vertex in $T_\sigma$, which we denote by $v_{-1}$. Note that the edge $\{v_{-1}, v_0\}$ is red because of the minimality of $P$, and since all blue edges in $T_\sigma$ are directed away from $r$ in the auxiliary tree.

The following lemma describes which modifications to the decomposition can be made if a minimal special path exists and how they change the legal order.

\begin{lemma}[Lemma 2.4 from \cite{ndtt}]	\label{lemma:specialPaths}
Let $\sigma = (R_1, \dots, R_t)$ be a legal order for $(T_1, \dots, T_k, F) \in \mathcal F$. \\
Furthermore, let $P = [v_0, v_1, \dots, v_l]$ be a minimal special path with respect to $\sigma$ and $(v_{l-1},v_{l})$ and let $i_0 := i_\sigma(v_0)$. Suppose that $v_{l-1}$ is not in the component of $R_{i_0} - v_0$ that contains $v_{-1}$. \\
Then there is a partition into forests $\mathcal{T}' = (T'_1, \dots, T'_k, F')$ of $G$ such that $T'_1, \dots, T'_k$ are spanning trees rooted at $r$ whose edges are directed to the respective parent vertex, the forest $F'$ exclusively consists of undirected edges and if $R'$ is the component containing $r$ in $F'$, we have that:
\begin{enumerate}
	\item $F' = \big( F + \{v_{l-1}, v_l\} \big) - \{v_{-1}, v_0\}$.
	\item $(v_0, v_{-1}) \in \bigcup_{j=1}^k E(T'_j)$.
	\item $\begin{aligned}[t]\bigcup_{j=1}^k \big\{\{u, v\} \: | \: (u, v) \in E(T'_j)\big\} = \bigcup_{j=1}^k \big\{\{u, v\} \: | \: (u, v) \in E(T_j)\big\} - \{v_{l-1}, v_l\} + \{v_0, v_{-1}\}.\end{aligned}$
	\item $\big\{(x, y) \in E(T'_j) \: | \: i_\sigma(x) < i_0\big\} = \big\{(x, y) \in E(T_j) \: | \: i_\sigma(x) < i_0\big\}$ for all $j \in \{1, \dots, k\}$.
	\item If $i_0 > 1$, then $R^* = R'$, $\mathcal{T}' \in \mathcal F$ and there exists a legal order $\sigma' = (R'_1, \dots, R'_{t'})$ for $\mathcal{T}'$ with $R'_j = R_j$ for all $j < i_0$ and $e(R'_{i_0}) < e(R_{i_0})$. Thus $\sigma' < \sigma$.
\end{enumerate}
\end{lemma}

The following lemma shows us how to choose a blue edge for which we can definitely find a (minimal) special path that satisfies the conditions of the previous lemma.

\begin{lemma}[cf.\ Corollary 2.5 in \cite{ndtt}] \label{lemma:goodMinPath}
Let $\sigma = (R_1, \dots, R_t)$ be a legal order for $(T_1, \dots, T_k, F) \in \mathcal F$ and $(x, y)$ a blue edge in this partition with $y \in \sigGreater{x}$. \\
Then there exists a minimal special path $[v_0, v_1, \dots, v_l]$ with respect to $\sigma$ and $(x, y)$ with $i_0 := i_\sigma(v_0)$ such that $x = v_{l-1}$ is not in the component $R_{i_0} - v_0$ that contains $v_{-1}$, and we can apply Lemma~\ref{lemma:specialPaths}.
\end{lemma}

Since we want to find smaller orders than $\sigma^*$ in our proofs to arrive at a contradiction, it is desirable that the fifth point holds in an application of Lemma~\ref{lemma:specialPaths}. In the event this is not the case, we can still gain more structure. We want to show this more formally with the next lemma:

\begin{lemma}		\label{lemma:R0R-1}
If Lemma~\ref{lemma:specialPaths} is applicable such that $\rho(F + \{x, y\}) = \rho^*$ holds, then $e(R_{i_0}) \leq d$ and therefore $i_0 > 1$. Moreover, in this case we have $\mathcal{T}' \in \mathcal F^*$ for the partition obtained.
\end{lemma}
\begin{beweis}
We use the notation of Lemma~\ref{lemma:specialPaths}. Then $R_{i_0}$ is split into two strictly smaller components in $F'$. If $\rho(F + \{x, y\}) = \rho^*$ and $e(R_{i_0}) > d$, then we had $\rho(F') < \rho^*$, a contradiction.

\end{beweis}

\begin{definition}
Let $\sigma$ be a legal order for $\mathcal{T} = (T_1, \dots, T_k, F) \in \mathcal F$.
Let $K$ be a red component of $\explSG$ and $C$ a child of $K$ with respect to $\sigma$. \\
Moreover, let $(x, y) \in E(T_i)$ be a blue edge such that $x \in V(K)$ and $y \in V(C)$. Then we say that \textit{$C$  is generated by $T_i$} or \textit{$C$ is generated by $(x, y)$}.
\end{definition}
We can now make our first structural statement about $\mathcal{T}^*$ and $\sigma^*$:

\begin{kor}[Corollary 2.5 from \cite{ndtt}] \label{kor:sumOfChildRelation} \phantom{bla} \\ 
Let $C$ be a child of $K$ with respect to $\sigma^*$ that is generated by $(x, y)$. Then $e(K) + e(C) \geq d$. 
\end{kor}
\vspace{1mm}
\begin{beweis}
We have $y \in \sigStarGreater{x}$ because $C$ is a child of $K$ with respect to $\sigma^*$, and because of Lemma~\ref{lemma:goodMinPath} we can apply Lemma~\ref{lemma:specialPaths}. 
Assume that $e(K) + e(C) < d$. Then $\rho(F^* + \{x, y\}) = \rho(F^*)$ and by Lemma~\ref{lemma:R0R-1} it follows that $i_0 > 1$ with the notation of Lemma~\ref{lemma:specialPaths}. Therefore, the fifth item of Lemma~\ref{lemma:specialPaths} applies and we obtain a legal order $\sigma$ for a partition $\mathcal{T} = (T_1, \dots, T_k, F) \in \mathcal F$ with $\sigma < \sigma^*$.
Because of Lemma~\ref{lemma:R0R-1} we have $\mathcal{T} \in \mathcal F^*$. This contradicts the minimality of $\sigma^*$.
\end{beweis}

We point out a particularly important special case of Corollary \ref{kor:sumOfChildRelation}.

\begin{kor}	\label{kor:smallNoSmallChildren}
Let $K$ be a red component of $H_{\mathcal{T}^*}$ that is small.
Then $K$ has no small children with respect to $\sigma^*$, i.e., there is no red component $C$ of $H_{\mathcal{T}^*}$ such that $C$ is a child of $K$ with respect to $\sigma^*$ and $C$ is small.
\end{kor}
\vspace{1mm}
\begin{beweis}
Assume such components $K, C$ did exist. Then $e(K) + e(C) \leq d - 1$ by Observation~\ref{beob:sizeOfSmall}, contradicting Corollary~\ref{kor:sumOfChildRelation}.
\end{beweis}

\section{Exchanging Edges}

In this section, we define a useful exchange operation and show how to reorient edges after the exchange to maintain the proper structure of the decomposition. We then use this exchange operation to show that $R^*$ does not have small children. After that we prove a lemma which shows three useful cases which can occur when trying to exchange edges. For this section we define $(T_1, \dots, T_k, F) \in \mathcal F$. 

\begin{definition}
Let $x, y$ be vertices in a component $K$ of $F$. Then we write $P_F(x, y) \subseteq K$ for the unique simple path from $x$ to $y$ in $F$. Similarly, for two vertices $x, y \in V$ such that $x$ is a descendant of $y$ in $T \in \{T_1, \dots, T_k\}$, let $P^T(x, y) \subseteq T$ be the unique directed path from $x$ to $y$ in $T$.
\end{definition}

\begin{definition}
Let $e \in E(T_i)$ and $e' \in E(F)$. If $(T_i - e) + e'$ is a spanning tree and $(F - e') + e$ is a forest (ignoring orientations), we say that $e'$ can be exchanged with $e$, and write that $e \leftrightarrow e'$ for $T_i$ and $F$.
We omit ``for $T_i$ and $F$'' if it is clear from context which forests are under consideration.
\end{definition}

The next lemma is obvious and we omit the proof, but it very usefully characterizes when $e \leftrightarrow e'$.

\begin{lemma} \label{lemma:exchange}
Let $u \in V - r$, $u'$ be the parent vertex of $u$ in $T_i$ and $e \in E(F)$.\\
Then, the following are equivalent:
\begin{enumerate}[(a)]
	\item $(u, u') \leftrightarrow e$.
	\item The edge $(u,u')$ lies in the unique cycle of $T_{i} + e$.
	\item One of the end vertices of $e$ is a descendant of $u$ in $T_i$ and the other is not.
\end{enumerate}
\end{lemma}

As we want to bound the number of small children in a red component $K$, it is useful to understand how edges of a tree $T$ can be exchanged with red edges.
\begin{lemma} \label{lemma:allDescendants}
Let $x, y \neq r$ be vertices in a red component $K$ of $\explSG$. Let $x'$ and $y'$ be the parent vertices of $x$ and $y$, respectively, in $T \in \{T_1, \dots, T_k\}$. Furthermore, let $P_F(x, y) = [x_1, \dots, x_n]$ with $x_1 = x$, $x_n = y$.
\begin{enumerate}[(a)]
	\item \label{i:dec1} If there exists an edge $\{x_i, x_{i+1}\}, 1 \leq i < n$ such that $(x, x') \leftrightarrow \{x_i, x_{i+1}\}$ holds, and we choose the minimal integer $i$ with this property, then for all $i' \leq i$ we have that $x_{i'}$ is a descendant of $x$ in $T$.
	\item \label{i:dec2} If there is no edge $\{x_i, x_{i+1}\}, 1 \leq i < n$, such that $(x, x') \leftrightarrow \{x_i, x_{i+1}\}$ holds, then all vertices in $P_F(x, y)$ (in particular, $y$) are descendants of $x$ in $T$.
\end{enumerate}
\end{lemma}
\vspace{1mm}
\begin{beweis}
When we speak of descendancy in this proof, it is always with respect to $T$. For part (a), choose $i$ minimally. Assume there was $i' \leq i$ such that $x_{i'}$ is not a descendant of $x$, and pick $i'$ to be the smallest integer with regards to this property. It follows that $i' > 1$ and $x_{i'-1}$ is a descendant of $x$. By Lemma~\ref{lemma:exchange} $(x, x') \leftrightarrow \{x_{i' - 1}, x_{i'}\}$, contradicting the minimality of $i$. \\
\ref{i:dec2}: If there exists a vertex $x_{i'}$ that is not a descendant of $x$, then we can analogously conclude that $(x, x') \leftrightarrow \{x_{i' - 1}, x_{i'}\}$ by choosing $i'$ minimally.
\end{beweis}

We now define an exchange operation which also will reorient the spanning trees so as to always have all vertices having a directed path towards $r$.
\begin{definition} \label{def:executeExchange}
Let $u \in V - r$ and $u'$ be the parent vertex of $u$ in $T_i$, let $e = \{v, w\} \in E(F)$ be a red edge, where $v$ is a descendant of $u$ in $T_i$ and $w$ is not a descendant of $u$ in $T_i$. We say that we obtain $T', F'$ by performing the exchange $(u, u') \leftrightarrow e$ in $T$ and $F$, where $T'$ and $F'$ are obtained in the following way: 
\begin{itemize}
    \item $T' = (T_i + (v, w)) - (u, u')$,
    \item $F' = (F - \{v, w\}) + \{u, u'\}$
    \item Reverse the orientation of the edges in $T'$ on $P^{T_{i}}(v,u)$.
\end{itemize}
\end{definition}

\begin{lemma}	\label{lemma:executeExchange}
Let $e = \{v,w\} \in E(F)$ be a red edge and $u'$ be the parent of a vertex $u \in T_{i}$. If we obtain the decomposition $(T_1, \dots, T_{i-1}, T', T_{i+1}, \dots, T_k, F')$ by performing the exchange $(u,u') \leftrightarrow e$ in $T_{i}$ and $F$, then $T_{1},\ldots,T_{i-1},T',T_{i+1},\ldots,T_{k}$ are spanning trees, $F'$ is an undirected forest, and the edges of $T'$ are oriented such that all vertices $v \in V(T')$ have a directed path from $v$ to $r$.
\end{lemma}
\vspace{1mm}
\begin{beweis}
By Lemma \ref{lemma:exchange} it suffices to show that the orientation of the edges in $T'$ is correct and therefore $P^{T'}(x, r)$ exists for every vertex $x \in V$. 
Note that $T_- := T_i - (u, u')$ consists of two components $U$ and $U'$ with $u \in V(U)$ and $u' \in V(U')$.
For each vertex of $x \in U'$ the path $P^{T_-}(x, r) = P^{T_i}(x, r)$ exists and for each $x \in U'$ the path $P^{T_-}(x, r)$ does not exist.
Now observe $U$ is a subtree of $T_i$ with root $u$ and $P^{T_i}(x, u)$ exists for all $x \in V(U)$. We see that $w \in U'$, since $w$ is not a descendant of $u$ in $T_i$.
For every vertex $x \in V(U)$ let $p_x$ be the first vertex on $P^{T_i}(x, u)$ that is also on the path $P^{T_i}(v, u)$ that is reoriented. We can now construct $P^{T'}(x, p_x)$ as:
\begin{align*}
P^{T_i}(x, p_x) \oplus \rP^{T_i}(v, p_x) \oplus [v, w] \oplus P^{T_i}(w, r).
\end{align*}
\end{beweis}

The described procedure enables us to enforce that $R^*$ does not have small children with respect to $\sigma^*$ and thus, the density around $R^*$ is high.

\begin{lemma} \label{lemma:rootNoChildren}
    The component $R^*$ does not have small children with respect to $\sigma^*$.
\end{lemma}
\begin{proof}
    Suppose to the contrary that $R^*$ has a child $C$ with respect to $\sigma^*$ that is generated by $(x, x')$. Let $x_i$ be the first vertex on $P_{F^*}(x, r) = [x_1, \ldots, x_n]$ that is not a descendant of $x$. This vertex exists and $i > 1$, since $x$ is a descendant of $x$ and $r$ is not. Then we have that $(x, x') \leftrightarrow \{x_{i-1}, x_i\}$. We obtain $T', F'$ by performing $(x, x') \leftrightarrow \{x_{i-1}, x_i\}$. By the way we chose $r$ the component $K_r$ of $r$ in $F'$ contains at least one edge if $d \leq k + 1$, and at least two edges if $d > k + 1$. The component $K_x$ of $x$ in $F'$ contains $\{x, x'\}$, but it does not contain $\{x_{i-1}, x_i\}$. Thus we have
    \[e(K_x) = e(R^*) + 1 + e(C) - 1 - e(K_r) < e(R^*).\]
    Since $K_r$ is a proper subgraph of $R^*$, we obtain a contradiction to the minimality of $\rho^*$.
\end{proof}

We want to bound the number of small children for every non-small red component. Thus, we now consider the situation that there are two distinct children generated by two edges of the same blue tree. The following lemma shows that we can always exchange at least one of the generating edges with an edge of the red path connecting the tails of the generating edges.

\begin{figure}[htp]
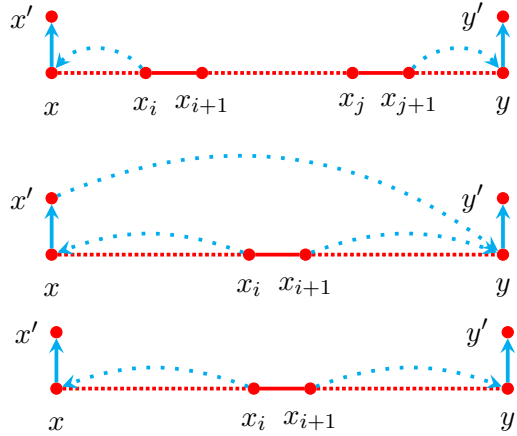

\center 
\caseOnePic
\caseTwoPic
\caseThreePic
	\caption{Case 1, 2 and 3 of Lemma \ref{sd:cases}}
	\label{fig:cases}
\end{figure}

\begin{lemma}		\label{sd:cases}
Let  and $\mathcal{T} = (T_1, \dots, T_k, F) \in \mathcal F$.
Let $K$ be a red component of $H_{\mathcal{T}}$, $T \in \{T_1, \dots, T_k\}$ and let $C_1, C_2$ be distinct children of $K$ with respect to $\sigma$ that are generated by edges $(x, x'), (y, y') \in E(T)$, respectively. Let $y$ not be a descendant of $x$ in $T$. Furthermore, let $P_F(x, y) = [x_1, \dots, x_n]$, so $x_1 = x$ and $x_n = y$.
Then one of the following three cases applies, which are depicted in Figure~\ref{fig:cases}:
\begin{enumerate}
	\item There is an edge $\{x_i, x_{i+1}\}$ and an edge $\{x_j, x_{j+1}\}$ such that $i < j$, $(x, x') \leftrightarrow \{x_i, x_{i+1}\}$ and $(y, y') \leftrightarrow \{x_j, x_{j+1}\}$. 
	Furthermore, for all vertices $x_{i'}$ with $i' \leq i$ we have that $x_{i'}$ is a descendant of $x$ in $T$ and $y \not\in V(P^T(x_{i'}, x))$. Likewise, for all vertices $x_{j'}$ with $j' \geq j + 1$ we have that $x_{j'}$ is a descendant of $y$ and $x \not\in V(P^T(x_{j'}, y))$ (therefore $x_{j'}$ is not a descendant of $x$ in $T$). Further, we have $E(P^T(x_i, x')) \cap E(P^T(x_{j+1}, y')) = \varnothing$. \\
	In this case we say that $x \caseOne y$ holds (for $(T, F))$ with edges $(x_i, x_{i+1}), (x_j, x_{j+1})$.
	
	\item We have that $x$ is a descendant of $y$ in $T$ and there is an edge $\{x_i, x_{i+1}\}$ such that  $(x, x') \leftrightarrow \{x_i, x_{i+1}\}$ holds. Furthermore, for all vertices $x_{i'}$ with $i' \geq i + 1$ we have that $x_{i'}$ is a descendant of $y$ in $T$ and $x \not\in V(P^T(x_{i'}, y))$ (therefore $x_{i'}$ is not a descendant of $x$ in $T$). Additionally, $x_i$ is a descendant of $x$ in $T$ and $E(P^T(x_i, x')) \cap E(P^T(x_{i+1}, y')) = \varnothing$. \\
	In this case we say that $x \caseTwo y$ holds (for $(T, F)$) with edge $(x_i, x_{i+1})$.
	
	\item The vertex $x$ is not a descendant of $y$ (and $y$ is not a descendant of $x$) in $T$ and there is an edge $\{x_i, x_{i+1}\}$ such that both $(x, x') \leftrightarrow \{x_i, x_{i+1}\}$ and $(y, y') \leftrightarrow \{x_i, x_{i+1}\}$ hold.
	Furthermore, for all vertices $x_{i'}$ with $i' \leq i$ we have that $x_{i'}$ is a descendant of $x$ in $T$ and $y \not\in V(P^T(x_{i'}, x))$ (therefore $x_{i'}$ is not a descendant of $y$). 
	Likewise, for all vertices $x_{j'}$ with $j' \geq i + 1$ we have that $x_{j'}$ is a descendant of $y$ and $x \not\in V(P^T(x_{j'}, y))$ (therefore $x_{j'}$ is not a descendant of $x$ in $T$). 
	We have $E(P^T(x_i, x')) \cap E(P^T(x_{i+1}, y')) = \varnothing$. \\ 
	In this case we say that $x \caseThree y$ holds (for $(T, F))$ with edge $(x_i, x_{i+1})$.
\end{enumerate}
\end{lemma}
\vspace{1mm}
\begin{beweis}
When we speak of descendancy in this proof, it is always with respect to $T$. We will first show that if we exclude the edge-disjoint path condition in each of the cases, then one of the three cases occurs, and then at the end of the proof, deduce the edge-disjoint path claim.\\
By Lemma~\ref{lemma:allDescendants}\ref{i:dec2} we have that there is always an edge $\{x_i, x_{i+1}\}$ such that $(x, x') \leftrightarrow \{x_i, x_{i+1}\}$ holds, for otherwise $y$ is a descendant of $x$.\\
We first assume that there is no edge $\{x_j, x_{j+1}\}$ such that $(y, y') \leftrightarrow \{x_j, x_{j+1}\}$. By Lemma~\ref{lemma:allDescendants}\ref{i:dec2} all vertices on $P_F(x, y)$ are descendants of $y$. \\
Now choose the maximum $i$ such that $(x, x') \leftrightarrow \{x_i, x_{i+1}\}$. 
Assuming that there is an $i' \geq i + 1$ such that $x \in V(P^T(x_{i'}, y))$, then $i' < n$.
If we choose the maximum $i'$ then $x \not\in V(P^T(x_{i' + 1}, y))$. Thus, $x_{i'}$ is a descendant of $x$ and $x_{i' + 1}$ is not a descendant of $x$. Therefore $(x, x') \leftrightarrow \{x_{i'}, x_{i'+1}\}$ holds by Lemma~\ref{lemma:exchange}, which contradicts $i$ being maximum.
Therefore $x \not\in V(P^T(x_{i'}, y))$ for all $i' \geq i + 1$. Since $x_{i+1}$ also is not a descendant of $x$ and $(x, x') \leftrightarrow \{x_i, x_{i+1}\}$ holds, $x_i$ must be a descendant of $x$ by Lemma~\ref{lemma:exchange}. Therefore, if this occurs, we are in Case 2. \\
Now assume that there exists an edge $\{x_j, x_{j+1}\}$ such that $(y, y') \leftrightarrow \{x_j, x_{j+1}\}$ holds. Choose the maximum $j$. By Lemma~\ref{lemma:allDescendants}\ref{i:dec1} all $x_{j'}, j' \geq j + 1$, are descendants of $y$.
If there is a $j' \geq j + 1$ such that $x \in V(P^T(x_{j'}, y))$ holds, then $x$ is a descendant of $y$. Choose the maximum $j'$ with this property. We have $j' < n$, $x_{j'}$ is a descendant of $x$ and $x \not\in V(P^T(x_{j' + 1}, y))$. Therefore $x_{j' + 1}$ is not a descendant of $x$. Because of Lemma~\ref{lemma:exchange} we have $(x, x') \leftrightarrow \{x_{j'}, x_{j'+1}\}$. Furthermore, since $j'$ is maximum,  for all $j'' \geq j' + 1$ we have $x \not\in V(P^T(x_{j''}, y))$. Thus again we are in Case 2.\\
We now assume that for all $j' \geq j + 1$ we have that $x_{j'}$ is a descendant of $y$ and $x \not\in V(P^T(x_{j'}, y))$.\\
Now choose the minimum $i$ such that $(x,x') \leftrightarrow (x_{i},x_{i+1})$. By Lemma \ref{lemma:allDescendants}\ref{i:dec1} all $x_{i'}$ where $i' \leq i$ are descendants of $x$ and furthermore, $y \not\in V(P^T(x_{i'}, x))$, for otherwise $y$ would be a descendant of $x$. Now we split into cases depending on if $i >j$, $i < j$ or $i = j$.
If $i > j$, then $x_i$ is a descendant of both $x$ and $y$. We also have $x \not\in V(P^T(x_i, y))$ and $y \not\in V(P^T(x_i, x))$, which is a contradiction because one of the vertices $x, y$ on the path $P^T(x_i, r)$ is reached before the other. If $i = j$ and $x$ is a descendant of $y$, then we are in Case 2. If $i = j$ and $x$ is not a descendant of $y$, then we are in Case 3. If $i < j$, then Case 1 applies. 
Thus we have shown that one of the three cases always applies up to the edge disjoint path conditions, which we show now. Let
\[(P_x, P_y) = 
	\begin{cases}
		(P^T(x_i, x'), P^T(x_{j+1}, y')) \text{, if Case 1 applies}\\
		(P^T(x_i, x'), P^T(x_{i+1}, y')) \text{, if Case 2 or Case 3 applies.}
	\end{cases}
\]
Assume that an edge $(v, v') \in E(P_x) \cap E(P_y)$ exists, then the paths $P^T(v, x), P^T(v, y)$ would exist. Therefore, either
$y \in V(P^T(v, x))$ or $x \in V(P^T(v, y)) \subseteq V(P_y)$ hold. The former can be ruled out because $y$ is not a descendant of $x$. The latter is also impossible in all three cases by their definition.
\end{beweis}

Note that if $x \caseOne y$ then it does not follow that $y \caseOne x$. Similarly with Case 2. 
However, if $x \caseThree y$ with edge $(u, v)$, then also $y \caseThree x$ with edge $(v, u)$.

As it will be useful later, we make the following notational definition:
\begin{notation}
If (for $T, F$) $x \caseOne y$ with two edges $(u_x, u_y)$ and $(v_x, v_y)$ or we have (for $T, F$) $y \caseOne x$ with edges $(v_y, v_x)$ and $(u_y, u_x)$, then we write that (for $T, F$) $x \caseOneUndir y$ holds with edges $(u_x, u_y)$ and $(v_x, v_y)$.
\end{notation}

\renewcommand\explSG{H_{\mathcal{T}^*}}

\section{Bounding small children - the cases \boldmath\texorpdfstring{$x \caseTwo y$}{x -2-> y} and \boldmath\texorpdfstring{$x \caseThree y$}{x -3- y}}

In this and the following section we show some situations in which we can utilize the exchanges we identified in the previous section to achieve a smaller legal order or decrease the residue function. By excluding these situations we obtain structure in $\mathcal T^*$. The results of this section will be summarized in Lemma~\ref{satz:caseTwoThree}. \\
Let $K$ be a red component of $\explSG$ and let $C_1, C_2$ be distinct small children of $K$ with respect to $\sigma^*$, which are generated by the edges $(x, x')$ and $(y, y')$ of the tree $T^* := T^*_\alpha \in \{T^*_1, \dots, T^*_k\}$, respectively. Throughout this section, suppose for $T^*, F^*$ that $x \caseTwo y$ or $x \caseThree y$ with an edge $(u, v)$. We fix the following notation for ease of the reader:

\begin{notation}	\phantom{force line break}
\begin{itemize}
	\item Let $K_1$ be the component of $K - \{u, v\}$ that contains $x$, and let $K_2$ be the other component that contains $y$. Thus, $e(K_1) + e(K_2) + 1 = e(K)$.
	\item We obtain $T, F$ by performing the exchange $(x, x') \leftrightarrow \{u, v\}$ and have \\ 
			$\mathcal{T} := (T^*_1, \dots, T^*_{\alpha-1}, T, T^*_{\alpha+1}, \dots, T^*_k, F)$.\\
	\item Let $K' = (V(K_1) \cup V(C_1), \; E(K_1) \cup E(C_1) + \{x, x'\})$.
	\item Let $K'' = (V(K_2) \cup V(C_2), \; E(K_2) \cup E(C_2) + \{y, y'\})$.
\end{itemize}
\end{notation}

\begin{figure}[htp]
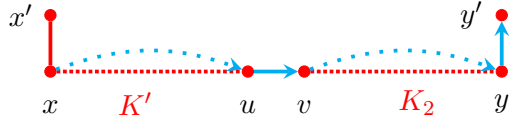

\center \caseThreePicAugm
\caption{The tree $T$ and the components $K', K_2$ of $F$.}
\label{fig:caseTwoThreeAugm}
\end{figure}

If we consider $\mathcal T$ with its red components $K'$ and $K_2$, we see that in some cases we obtain a smaller legal order if $e(K'), e(K_2) < e(K)$. However, exchanging an edge reorients a possibly non-trivial blue path, which could change the legal order: if there is a vertex $p$ on this path with $\iSig(p) < \iSig(K)$, then the red component of $p$ might lose a child and also might get a new child, both generated by edges that are incident to $p$. If this happens the legal order could increase at position $\iSig(p) + 1$. Fortunately, $\sigma^*$ still is intact for all indices that are less or equal than $\iSig(p)$, if we choose $p$ minimizing $\iSig$. We now first discuss briefly and informally how we receive a smaller legal order.

We first obtain $\mathcal T$ by the exchange described at the start of the section. This reorients $P^{T^*}(u, x)$. Let $p$ be a vertex on this path with $\iSig(p) = \iSig(P^{T^*}(u, x))$. As it can be seen in Figure \ref{fig:caseTwoThreeAugm}, there is a blue path from $p$ to $y$ now. Thus, an augmentation using a minimal special path with respect to $(y, y')$ goes back to a component $L$ with $\iSig(L) \leq \iSig(p)$. Thus, we get a smaller legal order. We now prove a technical lemma which will help us formalize the above (we use different notation in the next lemma to avoid confusion with the already defined notation in this section).

\begin{lemma}
\label{lemma:smallerByPath}
Let $L$ be a red component of $H_{\mathcal{T}^*}$ and let $C$ be a child of $L$ with respect to $\sigma^*$, which is generated by $(a, a') \in E(T^*)$, $T^* \in \{T^*_1, \dots, T^*_k\}$. Let $\mathcal S = (S_1, \dots, S_k, \Phi) \in \mathcal F$. Let $S \in \{S_1, \dots, S_k\}$, and suppose that $(a, a') \in E(S)$,  $\Phi + \{a, a'\}$ is a forest and $\rho(\Phi + \{a, a'\}) = \rho^*$.  Let $p \in V(R^*_{i_p})$ with $i_p \in \{1, \dots, t^*\}$ and suppose a legal order $\sigma = (R_1, \dots, R_t)$ exists for $\mathcal S$ with $R_j = R^*_j$ for all $j < i_p$ and $e(R_{i_p}) \leq e(R^*_{i_p})$. Furthermore, suppose $a \in V(R_{i_a})$ for an $i_a \geq i_p$ (in particular, $\iSig(L) \geq i_p$). Then there is no path $P^S(p, a)$.
\end{lemma}
\vspace{1mm}
\begin{beweis}
Assume $P^S(p, a)$ exists. Then $a \in V(H_{\mathcal S})$ and for an arbitrary minimal special path $[v_0, \dots, v_l]$ with respect to $\sigma$ and $(a, a')$ with $i_0 := i_\sigma(v_0)$ we have that $i_0 \leq i_p$ because of $P^S(p, a)$. We now will argue that we can apply Lemma \ref{lemma:specialPaths}.\\
If $i_a > i_0$, then $a \notin V(R_{i_0})$ and we can apply Lemma~\ref{lemma:specialPaths}. Otherwise we have $i_a = i_0 = i_p$. Since $\Phi + \{a, a'\}$ is a forest, we have $a' \notin V(R_{i_a})$. As $(R_{1},\ldots,R_{t})$ is the same as $\sigma^{*}$ until $i_{p}$, and $i_{a} = i_{p}$, and $a'$ lies in a child of $L$, it follows that $a' \in \sigma(>\!a)$. Thus by Lemma~\ref{lemma:goodMinPath}, we can again apply Lemma \ref{lemma:specialPaths}.

Since $\rho(\Phi + \{a, a'\}) = \rho^*$ and Lemma \ref{lemma:R0R-1} it follows from the fifth point of Lemma \ref{lemma:specialPaths} that $i_0 > 1$ holds and a legal order $\sigma' = (R'_1, \dots, R'_{t'})$ exists for a partition $(S'_1, \dots, S'_k, \Phi') \in \mathcal F^*$ with $R'_j = R_j = R^*_j$ for all $j < i_0 \leq i_p$ and $e(R'_{i_0}) < e(R_{i_0}) \leq e(R^*_{i_0})$. Thus, $\sigma' < \sigma^*$, which contradicts the minimality of $\sigma^*$.
\end{beweis}

Now we can determine in which cases two small children of a component generated by the same tree can occur:

\begin{lemma}	\label{satz:caseTwoThree}
Let $K$ be a red component of $\explSG$ and let $C_1, C_2$ be two distinct small children of $K$ with respect to $\sigma^*$ that are generated by the edges $(x, x')$ and $(y, y')$ of the tree $T^* \in \{T^*_1, \dots, T^*_k\}$, respectively. \\
Let $\{u, v\} \in E(P_{F^*}(x, y))$ and let $K_1$ be the component of $K - \{u, v\}$ that contains $x$, and let $K_2$ be the other component that contains $y$. \\
If $x \caseTwo y$ or $x \caseThree y$ with edge $(u, v)$, then one of the following cases holds:
\begin{enumerate}
	\item $e(K_1) = 0, e(C_2) = 1$,
	\item $e(K_2) = 0, e(C_1) = 1$,
\end{enumerate}

\end{lemma}
\begin{beweis}
Assume to the contrary that both $e(K_1) \geq e(C_2)$ and $e(K_2) \geq e(C_1)$. By Lemma \ref{lemma:rootNoChildren} we have that $K \neq R^*$. We start with showing that adding $\{y, y'\}$ to $F$ does not increase the residue function. We have:
\[e(K') = e(K_1) + e(C_1) + 1 = e(K) - e(K_2) + e(C_1) \leq e(K)\]
and analogously, $e(K'') \leq e(K).$ Thus, we have $\rho(F + \{y, y'\}) = \rho^*$.

First assume that $i_p := \iSig(P^{T^*}(u, x)) < \iSig(K) =: i_K$. The path $P^{T^*}(u, x)$ is the path that is reoriented in $T$. Let $p \in V(P^{T^*}(u, x))$ with $\iSig(p) = i_p$. Then there is a legal order $\sigma = (R_1, \dots, R_t)$ for $\mathcal{T}$ with $R_j = R^*_j$ for all $j \leq i_p$.
Because of the path
$P^T(p, y) = \rP^{T^*}(u, p) \oplus [u, v] \oplus P^{T^*}(v, y)$
Lemma~\ref{lemma:smallerByPath} provides a contradiction.\\

Therefore, one has $\iSig(P^{T^*}(u, x)) \geq \iSig(K) = i_K$. Then there is a legal order $\sigma = (R_1, \dots, R_t)$ for $\mathcal{T}$ with $R_ j= R^*_j$ for all $j < i_K$ and $R_{i_K} \in \{K', K_2\}$. Hence $e(R_{i_K}) \leq e(R^*_{i_K})$ and $i_\sigma(y) \geq i_K$. \\
Let $p := y$ if $R_{i_K} = K_2$, and $p := u$ if $R_{i_K} = K'$. Then the path 
\[P^T(p, y) = 
	\begin{cases}
		[y] \text{, if } R_{i_K} = K_2 \\
		 [u, v] \oplus P^{T^*}(v, y) \text{, if } R_{i_K} = K'
	\end{cases}
\]
exists and we contradict Lemma~\ref{lemma:smallerByPath}.
\end{beweis}

\section{Bounding small children - The case \boldmath\texorpdfstring{$x \caseOneUndir y$}{x -1- y}}
We now turn to the case $x \caseOneUndir y$ and proceed similarly to the previous section. 
The results will be summarized in Lemma~\ref{satz:caseOne}.

Let $K$ be a red component of $\explSG$ and let $C_1, C_2$ be distinct small children of $K$ with respect to $\sigma^*$ that are generated by the edges $(x, x')$ and $(y, y')$ of the tree $T^* \in \{T^*_1, \dots, T^*_k\}$, respectively. Thus $K \neq R^*$ by Lemma \ref{lemma:rootNoChildren}. Furthermore, suppose that $x \caseOneUndir y$ for $(T^*, F^*)$ with the edges $(u_x, u_y), (v_x, v_y)$ and (without loss of generality)  $\iSig(P^{T^*}(u_x, x)) \geq \iSig(P^{T^*}(v_y, y))$. 

\begin{notation} \phantom{break line}
\begin{itemize}	
	\item Let  $K_1$ be the component of  $K - \{u_x, u_y\}$ that contains $x$.
	\item Let $K_2$  be the component of $K - \{v_x, v_y\}$ that contains $y$.
	\item Let $K_3$  be the component of $\big(K - \{u_x, u_y\}\big) - \{v_x, v_y\}$ that contains neither $x$ nor $y$.
	\item Let $i_K := \iSig(K)$.
\end{itemize}
\end{notation}
Thus $e(K) = e(K_1) + e(K_2) + e(K_3) + 2$.

\begin{lemma}		\label{lemma:5}
If $i_K \leq \iSig(P^{T^*}(v_y, y))$, then $e(K_1) = 0$ and $e(C_2) = 1$.
\end{lemma}
\vspace{1mm}
\begin{beweis}
Assume to the contrary that $e(K_1) \geq e(C_2)$.
We obtain $T, F$ by performing the exchange $(y, y') \leftrightarrow \{v_x, v_y\}$. Let $\mathcal{T} := (T^*_1, \dots, T^*_{\alpha-1}, T, T^*_{\alpha+1}, \dots, T^*_k, F)$.\\
In $F$ there are the components
$K' = (V(K_1) \cup V(K_3), \: E(K_1) \cup E(K_3) + \{u_x, u_y\})$ and 
$K'' = (V(K_2) \cup V(C_2), E(K_2) \cup E(C_2) + \{y, y'\})$, for which we have $e(K') < e(K)$
and
\begin{align*}
	e(K'') &= e(K_2) + e(C_2) + 1 \\
	&= e(K) - e(K_1) - e(K_3) + e(C_2) - 1\\
	&\leq e(K) - e(K_3) - 1 \\
	&< e(K).
\end{align*}

\begin{figure}[h]
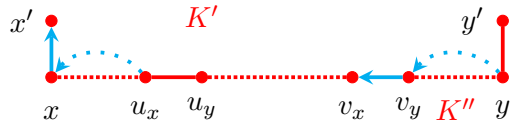

\center
\caseOnePicAugmFalse
\caption{$T$ and the components $K', K''$ of $F$ in Lemma~\ref{lemma:5}.}
\end{figure}

Thus $\rho(F) = \rho(F^*)$. 
Since $K \neq R^*$ and $i_K \leq \iSig(P^{T^*}(v_y, y))$ there is a legal order $\sigma = (R_1, \dots, R_t)$ for $\mathcal{T} \in \mathcal F^*$ with the property $R_j = R^*_j$ for all $j < i_K$ and $R_{i_K} \in \{K', K''\}$ and therewith $e(R_{i_K}) < e(R^*_{i_K})$. Thus, $\sigma < \sigma^*$, which is a contradiction.
\end{beweis}

Now we show the above lemma in a general setting, where the reorientation of a path when performing an exchange for $(x, x')$ or $(y, y')$ could change the legal order at an index $<i_K$. This will summarize the progress of this section.

\begin{lemma}	\label{satz:caseOne}
Let $K$ be a red component of  $\explSG$ and let $C_1, C_2$ be two distinct small children of $K$ with respect to $\sigma^*$, which are generated by the edges $(x, x')$ and $(y, y')$ of the tree $T^* \in \{T^*_1, \dots, T^*_k\}$, respectively. Furthermore, suppose that $x \caseOneUndir y$ for $(T^*, F^*)$ with the edges $(u_x, u_y), (v_x, v_y)$ and (without loss of generality) $\iSig(P^{T^*}(u_x, x)) \geq \iSig(P^{T^*}(v_y, y))$. Let $K_1$ be the component of $K - \{u_x, u_y\}$ that contains $x$.  \\
Then $e(K_1) = 0$ and $e(C_2) = 1$.
\end{lemma}
\begin{beweis}

Assume that $e(K_1) \geq e(C_2)$. Lemma~\ref{lemma:5} implies that
$i_p := \iSig(P^{T^*}(v_y, y)) < i_K$. 
This time we obtain $T, F$ by performing the exchange
 $(x, x') \leftrightarrow \{u_x, u_y\}$. 
Let $\mathcal{T} := (T^*_1, \dots, T^*_{\alpha-1}, T, T^*_{\alpha+1}, \dots, T^*_k, F)$ and we have $\mathcal{T} \in \mathcal F$.

\begin{figure}[h]
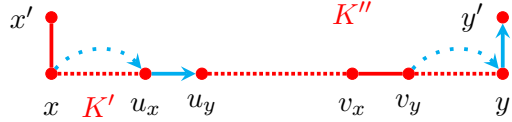

\center
\caseOnePicAugmTrue
\caption{$T$ and the components $K', K''$ of $F$ in Lemma \ref{satz:caseOne}.}
\end{figure}

We see that $F + \{y, y'\}$ has the components
\[K' = (V(K_1) \cup V(C_1), \: E(K_1) \cup E(C_1) + \{x, x'\})\]
and
\[K'' = (V(K_2) \cup V(K_3) \cup V(C_2), \: E(K_2) \cup E(K_3) \cup E(C_2) + \{v_x, v_y\} + \{y, y'\}),\]
for which
\begin{align*}
e(K') &= e(K_1) + e(C_1) + 1 \\
	&= e(K) - e(K_2) - e(K_3) + e(C_1) - 1 \\
	&\leq e(K)
\end{align*}
and
\begin{align*}
e(K'') &= e(K_2) + e(K_3) + e(C_2) + 2 \\
	&= e(K) - e(K_1) + e(C_2) \\
	&\leq e(K).
\end{align*}
Thus $\rho(F + \{y, y'\}) = \rho^*$. Since $i_p \leq \min\{ i_K, \iSig(P^{T^*}(u_x, x)) \}$ there is a legal order $\sigma = (R_1, \dots, R_t)$ for $\mathcal{T}$ with the property $R_j = R^*_j$ for all $j \leq i_p$. \\
Let $p \in V(P^{T^*}(v_y, y))$ with $\iSig(p) = i_p$.
Since $i_p < i_K$ we have $y \in \sigma(>p)$.
Since $E(P^{T^*}(u_x, x')) \cap E(P^{T^*}(v_y, y')) = \varnothing$ there exists a path $P^T(v_y, y') = P^{T^*}(v_y, y')$ and by the path
$P^T(p, y) \subseteq P^T(v_y, y')$ and Lemma~\ref{lemma:smallerByPath} we arrive at a contradiction.
\end{beweis}

\section{Proof of the Conjecture for \boldmath\texorpdfstring{$d \leq k + 1$}{d <= k+1}}
We can now, by Lemmas \ref{satz:caseTwoThree} and \ref{satz:caseOne}, bound the number of small children per component in the case $d \leq k + 1$. In this way, we find the contradiction that the density of $\explSG$ is too high.

\begin{kor}		\label{kor:oneZeroChild}
Let $K$ be a red component of  $\explSG$. Then for every $i \in \{1, \dots, k\}$, 
the number of children $C$ of $K$ with respect to $\sigma^*$ which are isolated vertices and which are generated by $T^*_i$ is at most 1.
\end{kor}
\vspace{1mm}
\begin{beweis}
Note that $R^*$ does not have any small children by Lemma \ref{lemma:rootNoChildren}. Thus let $K \neq R^*$. If there are two distinct small children of $K$ generated by $T^*_i$, then by Lemma \ref{satz:caseOne} and \ref{satz:caseTwoThree} one of them must contain an edge and thus $d > k + 1$.
\end{beweis}

We are now ready to prove the Strong Nine Dragon Tree Conjecture when $d \leq k+1$. Recall  Corollary~\ref{kor:smallNoSmallChildren} that all small red components of $\explSG$ are children of non-small components. 

\begin{notation}
Denote the set of red components of  $\explSG$ that are not small by  $\mathcal{K}$. In an arbitrary fashion we assign each small components to exactly one of its parents in  $\mathcal{K}$. \\
Let  $K \in \mathcal{K}$ and $C_1, \dots, C_q$ the small children of $K$ that were assigned to  $K$ .\\
Then
$\mathcal{C}(K) := \{C_1, \dots, C_q\}$, $\mathcal{C}_l(K) = \{C \in \mathcal{C}(K) | e(C) = l\}$ and 
\[\KC := \big(V(K) \cup \bigcup_{C \in \mathcal{C}(K)} V(C), \;E(K) \cup \bigcup_{C \in \mathcal{C}(K)} E(C)\big).\]
\end{notation}

\begin{beob}
We have $V(\explSG) = \dot\bigcup_{K \in \mathcal{K}} V(\KC)$ and $E_r(\explSG) = \dot\bigcup_{K \in \mathcal{K}} E(\KC)$.
\end{beob}

\begin{lemma}
If $d \leq k + 1$, then for every $K \in \mathcal{K}$:
\[\frac{e(\KC)}{v(\KC)} \geq \frac{d}{d+k+1}.\] 
In particular, Theorem \ref{satz:mySndtc} holds when $d \leq k+1$.
\end{lemma}
\vspace{1mm}
\begin{beweis}
If $e(K) < d$, then  by Corollary~\ref{kor:sumOfChildRelation} $K$ does not have small children and since $K$ is not small, $\frac{e(\KC)}{v(\KC)} \geq \frac{d}{d+k+1}$.  Now let $e(K) \geq d$. Then
\[
    \frac{e(\KC)}{v(\KC)}
    = \frac{e(K)}{e(K) + 1 + \sum\limits_{C \in \mathcal{C}(K)} v(C)}
    \overset{Cor.~\ref{kor:oneZeroChild}}{\geq} \frac{e(K)}{e(K) + 1 + k}
    \geq \frac{d}{d+k+1}
\]
and we get
\[
	\frac{e_r(\explSG)}{v(\explSG) - 1}
		> \frac{e_r(\explSG)}{v(\explSG)}
		= \frac{\sum_{K \in \mathcal{K}} e(\KC)}{\sum_{K \in \mathcal{K}} v(\KC)}
		\geq \frac{d}{d+k+1},
\]
which is a contradiction to Observation~\ref{beob:densityExplSG}. Therefore Theorem~\ref{satz:hilfsbeh} is proved when $d \leq k + 1$.
\end{beweis}

\section{Case \boldmath\texorpdfstring{$k+1 < d < 3(k+1)$}{k + 1 < d < 3(k+1)}}

If we were also able to show for $k + 1 < d$ that each red component of $\explSG$ had at most $k$ small children, we would have proven the Strong Nine Dragon Tree Conjecture for all the other cases. As we are not able to do that, our goal will be to show that each red component of $\explSG$ has at most $2k$ small children each containing at most one red edge.

\subsection{Elimination of the cases \boldmath\texorpdfstring{$x \caseOneUndir y$}{x -1- y} and \boldmath\texorpdfstring{$x \caseThree y$}{x -3- y}} \label{subsection:elimOneThree}

In this subsection we show for $k + 1 < d < 3(k+1)$ that there cannot be two vertices $x, y$ and two small children of a component $K$ with $e(K) \geq 3$ that are each generated by edges $(x, x'), (y, y')$ of a spanning tree such that either $x \caseOneUndir y$ or $x \caseThree y$ holds.\\
In the following lemma we consider the case where an edge swap reverses a trivial path.

\begin{lemma}	\label{lemma:decOfX}
Let $K$ be a red component of $\explSG$ containing a red edge $e = \{x, y\}$ such that $C$ is a small child of $K$ with respect to $\sigma^*$ that is generated by $(x, x') \in E(T^*_i)$. Let $K_1$ be the component of $K - e$ containing $x$. Furthermore let $e(K_1) \leq e(K) - e(C) - 2$. \\
Then $y$ is a descendant of $x$ in $T^*_i$.
\end{lemma}
\vspace{1mm}
\begin{beweis}
Note that $K \neq R^*$ because of Lemma \ref{lemma:rootNoChildren}. We assume $y$  is not a descendant of $x$. It follows that we have that $(x, x') \leftrightarrow e$ by Lemma \ref{lemma:exchange}. We obtain $T, F$ by performing this exchange. We write $\mathcal{T} := (T^*_1, \dots, T^*_{i-1}, T,  T^*_{i+1}, \dots, T^*_k, F)$ and let $K_2$ be the component of $K - e$ containing $y$.

Note $K' = (V(K_1) \cup V(C), E(K_1) \cup E(C) + \{x, x'\})$ and $K_2$ are components in $F$ and it holds
\[e(K') = e(K_1) + e(C) + 1 \leq e(K) - 1 < e(K)\]
as well as $e(K_2) < e(K)$.
Thus $\rho(F) = \rho^*$ and because $T = T^*_i + (x, y) - (x, x')$ there exists a legal order $\sigma = (R_1, \dots, R_t)$ for $\mathcal{T}$ with $R_j = R^*_j$ for all $j < \iSig(K)$ and $R_{\iSig(K)} \in \{K', K_2\}$, which leads to the desired contradiction $\sigma' < \sigma^*$.
\end{beweis}

\begin{kor}		\label{kor:f1types}
Let $K$ be a red component of $\explSG$ and $C_1, C_2$ be distinct, small children of $K$ which are generated by the edges $(x, x')$ and $(y, y')$ of the tree $T^* \in \{T^*_1, \dots, T^*_k\}$ respectively. Furthermore, assume that $y$ is not a descendant of $x$ and let $e(K) \geq 3$. Then it holds $x \caseTwo y$ for $(T, F)$ with edge $(u, v)$. \\
Further, let $K_2$ be the component of $K - \{u, v\}$ containing $y$. 
Then $e(K_2) = 0$ and $e(C_1) = 1$.
\end{kor}
\begin{beweis}
By Lemma \ref{lemma:rootNoChildren} we have that $K \neq R^*$.
Let $P_{F^*}(x, y) = [x_1, \dots, x_n]$ and consider the subgraph $X = (\{x\}, \varnothing)$ with
$e(X) = 0 \leq e(K) - e(C_1) - 2$. By Lemma \ref{lemma:decOfX}, it follows that $x_2$ is a descendant of $x$. Analogously, we can derive that $x_{n-1}$ is a descendant of $y$ in $T^*$.
Using the notation of Lemma \ref{satz:caseOne} and Lemma \ref{satz:caseTwoThree}, respectively, this implies that $e(K_1) \geq 1$ and if $x \caseThree y$ or $x \caseOneUndir y$, then also $e(K_2) \geq 1$, by the definition of ``$\caseOne$'', ``$\caseTwo$'' and ``$\caseThree$''. 
By Lemma \ref{satz:caseOne} and Lemma \ref{satz:caseTwoThree} the lemma follows.
\end{beweis}

\subsection{Case \boldmath\texorpdfstring{$d=3$}{d = 3} and \boldmath\texorpdfstring{$e(K) = 2$}{e(K) = 2}}
Corollary \ref{kor:f1types} will be very useful to bound the number of small children of a red component. Unfortunately, it only holds for $e(K) \geq 3$. In this subsection, we will consider the case $e(K) = 2$ and $d=3$. In case of $d \geq 4$, we have already proven through 
Corollary \ref{kor:sumOfChildRelation} that a red component containing only two edges cannot have small children, since we assumed $2 \leq k + 1 < d < 3(k+1)$.

\begin{lemma}	\label{lemma:e2d3}
Let $d = 3$ and let $K$ be a red component of $\explSG$ with $e(K) = 2$. Then there are at most two distinct small children $C_1, C_2$ of $K$ with respect to $\sigma^*$ which are both generated by a spanning tree $T^* := T^*_\alpha \in \{T^*_1, \dots, T^*_k\}$. Furthermore it holds $e(C_1) = e(C_2) = 1$.
\end{lemma}
\vspace{1mm}
\begin{beweis}
Note $e(C) = 1$ holds for each small child $C$ of $K$ because of Corollary \ref{kor:sumOfChildRelation}. Now we assume there are three children which are generated by $T^*$ and each has exactly one edge. Then there are two vertices $x, y \in V(K)$ with $\{x, y\} \in E(K)$ and two distinct, small children $C_1, C_2$ of $K$ which are generated by $(x, x'), (y, y') \in E(T^*)$ and for which $e(C_1) = e(C_2) = 1$ holds.
Because of $e(K) < d$ it holds $K \neq R^*$. \\
We can assume that $y$ is not a descendant of $x$. Therefore $(x, x') \leftrightarrow \{x, y\}$ holds due to Lemma \ref{lemma:exchange} and we receive
$T = T^* - (x, x') + (x, y), F = F^* - \{x, y\} + \{x, x'\}$ by performing the exchange $(x, x') \leftrightarrow \{x, y\}$. Let $\mathcal{T} := (T^*_1, \dots, T^*_{\alpha-1}, T, T^*_{\alpha+1}, \dots, T^*_k, F)$ and $i_K := \iSig(K)$. \\
Let $K_1$ be the component of $K - \{x, y\}$ containing $x$ and let $K_2$ be the other component containing $y$. Additionally, let $K' = (V(K_1) \cup V(C_1), \; E(K_1) \cup E(C_1) + \{x, x'\})$. Then $K', K_2$ are components of $F$ and it holds 
$e(K') = e(K_1) + e(C_1) + 1 \in \{2, 3\}$ and $e(K_2) \in \{0, 1\}$. Thus $\mathcal{T} \in \mathcal F^*$. \\
Furthermore, there is a legal order $\sigma = (R_1, \dots, R_t)$ for $\mathcal{T}$ with $R_j = R^*_j$ for all $j < i_K$ and $R_{i_K} \in \{K', K_2\}$. \\
If we could choose $R_{i_K} = K_2$, $\sigma < \sigma^*$ would hold leading to a contradiction.

Thus we assume that $R_{i_K}$ cannot be chosen to be $K_2$. Therefore there must be a blue edge $(v, v')$ with $v' \in V(K_1)$ in $T_\sigma$ (and $T_{\sigma^*}$). Due to the edges $(x, y), (y, y') \in E(T)$ we can also choose $\sigma$ such that $R_{i_K + 1} = K_2$ and $R_{i_K + 2} = C_2$ and finally $y' \in \sigma(>y)$ holds.

We consider an arbitrary minimal special path $P = [v_0, v_1, \dots, v_l]$ with respect to $\sigma$ and $(y, y')$. Due to the path $P^T(x, y') = [x, y, y']$ it holds $i_0 := i_\sigma(v_0) \leq i_K$. \\
For the component $K'' := (V(K_2) \cup V(C_2), \; E(K_2) \cup E(C_2) + \{y, y'\})$ of $F + \{y, y'\}$ it also holds $e(K'') = e(K_2) + e(C_2) + 1 \leq 3$ and thus $\rho(F + \{y, y'\}) = \rho^*$.
Using Lemma \ref{lemma:specialPaths} and \ref{lemma:R0R-1} we conclude that there is a legal order $\sigma' = (R'_1, \dots, R'_{t'})$ for a partition from $\mathcal F^*$ with $R'_j = R_j = R^*_j$ for all $j < i_0$ and $e(R'_{i_0}) < e(R_{i_0})$.\\
For $i_0 < i_K$, we receive the contradiction $\sigma' < \sigma^*$ because of $e(R_{i_0}) = e(R^*_{i_0})$. \\
For $i_0 = i_K$, it must hold $v' \neq x$, or otherwise it would hold $[v, v', y, y'] \leq P$ contradicting the minimality of $P$. Thus it holds $K_1 = \{v', x\}$, $v_0 = x$, $v_{-1} = v'$ and the component of $K' - \{v_{-1}, v_0\}$ containing $v_{-1}$ only consists of this vertex.
Now we can choose $\sigma'$ such that $R'_{i_K} = (\{v'\}, \varnothing)$ holds and by that we once again achieve the contradiction $\sigma' < \sigma^*$.
\end{beweis}

\subsection{Elimination of \boldmath\texorpdfstring{$x \caseTwo y \caseTwo z$}{x -2-> y -2-> z}}
In this subsection we consider the remaining cases where there are three small children of $K \neq R^*$ generated by one tree. This will reduce to $x \caseTwo y \caseTwo z$. More formally, let $d \geq 3$, $K$ be a red component of $\explSG$ with $e(K) \geq 3$,  $C_x, C_y, C_z$ be distinct, small children of $K$ with respect to $\sigma^*$ which are generated by the edges $(x, x'), (y, y'), (z, z')$ of the tree $T^* \in \{T^*_1, \dots, T^*_k\}$, respectively. 
The goal of this subsection is to obtain a contradiction in order to bound the number of small children of a red component to $2k$. By Corollary \ref{kor:f1types} we may assume that without loss of generality it holds $x \caseTwo y$ with edge $(u, y)$ and $y \caseTwo z$ with edge $(v, z)$ for two vertices $u, v \in V$.

\begin{lemma} \label{lemma:secondSwapPossible}
After performing the exchange $(x, x') \leftrightarrow \{u, y\}$ there (still) exists a path of $T$ from $v$ to $y$ and $y$ (still) is a descendant of $z$ in $T$. Thus we have $(y, y') \leftrightarrow \{v, z\}$ in $T$.
\end{lemma}
\vspace{1mm}
\begin{beweis}
We obtain $T, F$ by performing $(x, x') \leftrightarrow \{u, y\}$. 
During the exchange we have only reoriented or removed blue edges whose incident vertices are descendants of $x'$ in $T^*$. Thus $y$ still is a descendant of $z$ in $T$. 
Next, we want to show that $P^T(v, y)$ exists. If $V(P^{T^*}(v, y)) \cap V(P^{T^*}(u, x)) = \varnothing$, then $P^{T^*}(v, y) = P^T(v, y)$. Otherwise let $P^{T^*}(v, y) = [v_1, \dots, v_m]$. We choose $j$ such that $(v_j, v_{j+1}) \notin E(T)$ and $j$ is minimal. It holds $v_j \in V(P^{T^*}(u, x))$. Therefore the path
$[v_1, \dots, v_j] \oplus \rP^{T^*}(u, v_j) \oplus [u, y] = P^T(v, y)$ exists \ \\

\end{beweis}

\begin{figure}[!htb]
	\minipage{0.32\textwidth}
	\xyzPicOne
	\centering $T^*$ and $F^*$
	\endminipage\hfill
	\minipage{0.32\textwidth}
	\xyzPicTwo
	\centering $T$ and $F$
	\endminipage\hfill
	\minipage{0.32\textwidth}%
	\center\xyzPicThree
	\centering $T'$ and $F'$
	\endminipage
	\caption{$K$ before and after the swaps, if $P^{T^*}(u, x) \cap P^{T^*}(u, x) \neq \varnothing$}
\end{figure}

We obtain $T, F$ by performing $(x, x') \leftrightarrow \{u, y\}$ in $T^*, F^*$ and we obtain $T', F'$ by performing $(y, y') \leftrightarrow \{v, z\}$ in $T, F$. 

We now want to use the special paths argument for $(z, z')$. $F' + \{z, z'\}$ has the components
\begin{itemize}
    \item $K_y = (V(C_y) + y, \:E(C_y) + \{y, y'\})$,
    \item $K_z = (V(C_z) + z, \:E(C_z) + \{z, z'\})$ and
    \item $K_x = \big((V(K) \cup V(C_x)) \setminus \{y, z\}, \: \big(E(K) \cup E(C_x) + \{x, x'\}\big) \setminus \big\{\{u, y\}, \{v, z\}\big\}\big)$.
\end{itemize}

As it holds $e(K_y), e(K_z) \leq 2 \leq d$ and $e(K_x) \leq e(K)$ we have $\rho(F' + \{z, z'\}) = \rho^*$. Let $\mathcal P = V(P^{T^*}(u, x)) \cup V(P^T(v, y))$ be the set of vertices of the reoriented paths.

\begin{lemma}		\label{lemma:pathPZ}
For every vertex $p \in \mathcal P$ the path $P^{T'}(p, z)$ exists.
\end{lemma}
\begin{beweis}
For $p \in V(P^T(v, y))$ the path $P^{T'}(p, z) = \rP^T(v, p) \oplus [v, z]$ exists. \\
Now let $p \in V(P^{T^*}(u, x)) \setminus V(P^T(v, y))$. We do the same case distinction as in the proof of Lemma \ref{lemma:secondSwapPossible}: If $V(P^{T^*}(v, y)) \cap V(P^{T^*}(u, x)) = \varnothing$, the path $P^T(p, u) \oplus [u, y] \oplus P^T(y, v) \oplus [v, z] = P^{T'}(p, z)$ exists in $T'$. Otherwise we choose $v_j$ like in the proof of Lemma \ref{lemma:secondSwapPossible}. Then $\rP^{T*}(v_j, p) = P^T(p, v_j) = P^{T'}(p, v_j)$ exists as well as $P^{T'}(v_j, z)$, since $v_j \in V(P^T(v, y))$.
\end{beweis}

\begin{lemma} \label{lemma:threeOneEdgeChildrenContradiction}
The assumption of $K$ having $C_x, C_y$ and $C_z$ as small children generated by the same tree causes a contradiction.
\end{lemma}
\begin{beweis}

The following procedure is similar to the end of the proof of Lemma \ref{satz:caseTwoThree}.\\
First, assume $\iSig(\mathcal P) \geq \iSig(K) =: i_K$. Then there is a legal order $\sigma' = (R'_1, ..., R'_{t'})$ for $\mathcal T'$ such that $R'_j = R^*_j$ for all $j < i_K$ and $R'_{i_K} \in \{K_x, K_y, K_z\}$. Thus $e(R'_{i_K}) \leq e(R^*_{i_K})$.
Let $p \in \{x, y, z\}$ be defined by $R'_{i_K} = K_p$. Now, the path 
$Q := \rP^{T^*}(u, x) \oplus [u, y] \oplus \rP^{T}(v, y) \oplus [v, z] \subseteq T'$
includes $P^{T'}(p, z)$ and further we have $\iSig(p) \leq \iSig(z)$. But the existence of $Q$ is a contradiction to Lemma \ref{lemma:smallerByPath}.

Thus, it has to be $\iSig(\mathcal P) < \iSig(K)$. Let $p \in V(\mathcal P)$ with $i_p := \iSig(p) = \iSig(\mathcal P)$. There is a legal order $\sigma' = (R'_1, ..., R'_{t'})$ for $\mathcal T'$ such that $R'_j = R^*_j$ for all $j \leq i_p$. Notice that again $Q$ contains $P^{T'}(p, z)$. But this is again a contradiction to Lemma \ref{lemma:smallerByPath}.
\end{beweis}

\subsection{Bounding \boldmath\texorpdfstring{$\KC$}{KC}}

In the following corollary we wrap up the results of the previous subsections, in particular Lemma \ref{lemma:e2d3} and Lemma \ref{lemma:threeOneEdgeChildrenContradiction}.

\begin{kor} \label{kor:leqTwoChildren1}
Let $K$ be a red component of  $\explSG$. Then for every $i \in \{1, \dots, k\}$, the number of small children $C$ of $K$ with respect to $\sigma^*$ which contain at most one edge and which are generated by $T^*_i$ is at most $2$.
\end{kor}

Finally, we can complete the proof of Theorem \ref{satz:hilfsbeh}.

\begin{lemma}
Let $k + 1 < d < 3(k + 1)$. Then for each $K \in \mathcal{K}$ it holds:
\[\frac{e(\KC)}{v(\KC)} \geq \frac{d + k}{d + 3k + 1}.\]
In particular, there is no minimal counterexample to Theorem \ref{satz:hilfsbeh}.
\end{lemma}
\begin{beweis}
Note that $d \geq 3$. Let $K \in \mathcal{K}$.\\
If $e(K) - 1 < d$, then $K$ has no small children by Corollary \ref{kor:sumOfChildRelation} and since $K$ is not small, we have $\frac{e(\KC)}{v(\KC)} = \frac{e(K)}{v(K)}\geq \frac{d + k}{d + 3k + 1}$.\\

If $e(K) = d - 1$, then $K$ by Corollary \ref{kor:sumOfChildRelation} has no small children consisting of exactly one vertex.
Due to Corollary \ref{kor:leqTwoChildren1} it holds $|\mathcal{C}(K)| \leq 2k$ and for all $C \in \mathcal{C}(K)$ it holds $e(C) = 1$.\\
We have $e(K)/v(K) = e(K)/(e(K) + 1) \geq 2/3$ and for $C \in \mathcal{C}(K)$ it holds
$e(C)/v(C) = 1/2$. Hence, we have
\begin{align*}
	\frac{e(\KC)}{v(\KC)}
		&= \frac{e(K) + \sum_{C \in \mathcal{C}(K)} e(C)}{e(K) + 1 + \sum_{C \in \mathcal{C}(K)} v(C)} \\
		&= \frac{e(K) + |\mathcal{C}(K)| \cdot 1}{e(K) + 1 + |\mathcal{C}(K)| \cdot 2} \\
		&\geq \frac{d - 1 + 2k}{d + 4k}\\
		&\geq \frac{d + k}{d + 3k + 1}.
\end{align*}

Now we consider the last case $e(K) \geq d$. Then $\mathcal{C}_0(K) \leq k$ holds because of Corollary \ref{kor:oneZeroChild} and $\mathcal{C}_1(K) \leq 2k - |C_0(K)|$ holds due to Corollary \ref{kor:leqTwoChildren1}.\\
For $x, y \in \mathbb N$ satisfying $x/y \leq 1$ it holds $\frac{x + 0}{y + 1} < x/y$ as well as $\frac{x + 0}{y + 1} < \frac{x + 1}{y + 2}$. Applying induction to this argument we get the smallest ``density'' in the following first inequality by maximizing $|\mathcal{C}_0(K)|$, which yields $|\mathcal{C}_0(K)| \leq k$ and $|\mathcal{C}_1(K)| \leq k$:

\begin{align*}
	\frac{e(\KC)}{v(\KC)}
		&= \frac{e(K) + |\mathcal{C}_1(K)| \cdot 1}{e(K) + 1 + |\mathcal{C}_0(K)| \cdot 1 + |\mathcal{C}_1(K)| \cdot 2} \\
		&\geq \frac{e(K) + |\mathcal{C}_1(K)| \cdot 1}{e(K) + 1 + k + |\mathcal{C}_1(K)| \cdot 2} \\
		&\geq \frac{d + |\mathcal{C}_1(K)| \cdot 1}{d + 1 + k + |\mathcal{C}_1(K)| \cdot 2} \\
		&\geq \frac{d + k}{d + 3k + 1}.
\end{align*}
The last inequality is true, since $d/(d+k+1) > 1/2$.

For $k + 1 < d < 3(k + 1)$ we conclude analogously to the case $d \leq k + 1$:
\[
	\frac{e_r(\explSG)}{v(\explSG) - 1} 
	    > \frac{e_r(\explSG)}{v(\explSG)} 
	    = \frac{\sum_{K \in \mathcal{K}} e(\KC)}{\sum_{K \in \mathcal{K}} v(\KC)}
		\geq \frac{d + k}{d+3k+1},
\]
which contradicts Observation \ref{beob:densityExplSG} and thus we have proven Theorem \ref{satz:hilfsbeh} for the second case, too.
\end{beweis}

\section{Thin Trees in planar \boldmath\texorpdfstring{$5$}{5}-edge-connected graphs}

In this section, we prove Theorem \ref{thintrees} following the approach of \cite{boundeddiameter}. We first apply Theorem \ref{satz:mySndtc} to planar graphs of girth at least five. As notation, for an embedded planar graph, we let $F(G)$ denote the set of faces, and $f(G)$ the number of faces of $G$. For a face $f$, we also let $|f|$ be the number of edges incident to $f$. 

\begin{lemma}
\label{girth5decomp}
If $G$ is a planar graph of girth at least five, then $G$ decomposes into two forests $T,F$ such that each component of $F$ has at most five edges. 
\end{lemma}

\begin{beweis}
It suffices to show that every planar graph of girth at least five has fractional arboricity at most $1 + \frac{2}{3}$, as then the result follows by Theorem \ref{satz:mySndtc} by plugging in $k=1$ and $d=4$. This is clear, if $v(G) \leq 2$, so let $v(G) \geq 2$. By Euler's formula, $v(G) -e(G) +f(G) =2$. Further, $2e(G) = \sum_{f \in F(G)} |f| \geq 5f(G)$. Thus $v(G) -e(G) + \frac{2}{5}e(G) = v(G) -\frac{3}{5}e(G) \geq 2$. Rearranging, we see that $\frac{5}{3} \geq \frac{e(G)}{v(G) - 2} > \frac{e(G)}{v(G) - 1}$.
\end{beweis}

For a planar graph $G$, we let $G^{*}$ denote the dual graph of $G$. We recall the following well-known observation that cycles and cuts are duals in planar graphs. 
\begin{beob}
If $G$ is a $5$-edge-connected planar graph, then $G^*$ is a simple graph of girth at least five.
\end{beob}

Now we are ready to prove the result. We assume basic knowledge about dual graphs and cut-cycle duality in planar graphs.

\begin{satz}
Every $5$-edge-connected planar graph has a $\frac{5}{6}$-thin tree.
\end{satz}

\begin{beweis}
Let $G$ be any $5$-edge-connected planar graph. By Lemma \ref{girth5decomp}, we have that $G^{*}$ decomposes into a tree $T'$ and a forest $F'$ such that each component of $F'$ has at most five edges. 
We obtain $F \subseteq G$ from $F'$ using the usual bijection $\varphi: E(G^*) \rightarrow E(G)$.

Now, consider a non-empty cut-set $S \subseteq E(G)$. Note that $\varphi^{-1}(S)$ is an even subgraph in $G^{*}$, and hence decomposes into edge-disjoint cycles. As $T'$ is acyclic $\varphi^{-1}(S)$ contains at least one edge of $F'$ and thus $S$ contains at least one edge of $F$. Thus $F$ is spanning and connected.

Furthermore, each cycle $C'$ of $\varphi^{-1}(S)$ has at most $\frac{5}{6}e(C')$ edges belonging to $F'$ since each component of $F'$ has at most $5$ edges. Thus it follows that at most $\frac{5}{6}|S|$ edges belong to $F$. Finally, we can obtain a $\frac{5}{6}$-thin tree from $F$ by removing edges on cycles until it is acyclic.
\end{beweis}

\begin{ack}
Both authors would like to thank the referees for suggestions which improved the presentation of the paper. The first author would like to thank Markus Blumenstock for introducing him to the topic, lengthy proof-readings, translation work and helpful discussions. The second author would like to thank Logan Grout for countless discussions on the Strong Nine Dragon Tree Conjecture.
\end{ack}
\printbibliography

\end{document}